\documentclass{amsart}

\usepackage{amssymb,amsfonts,amsmath}
\usepackage[all,arc]{xy}
\usepackage{enumerate}
\usepackage{mathrsfs}
\usepackage{multicol}
\usepackage{xcolor}
\usepackage{lipsum}
\usepackage{mdwlist}
\usepackage{mathtools}
\usepackage{comment}
\usepackage{tikz-cd}
\usepackage{bbm}
\usepackage{enumitem}
\usepackage{hyperref}

\newtheorem{thm}{Theorem}[section]
\newtheorem{cor}[thm]{Corollary}
\newtheorem{prop}[thm]{Proposition}
\newtheorem{lem}[thm]{Lemma}

\theoremstyle{definition}
\newtheorem{defn}[thm]{Definition}
\newtheorem{exmp}[thm]{Example}
\newtheorem{rmk}[thm]{Remark}
\newtheorem{conj}[thm]{Conjecture}

\DeclareMathOperator{\Sym}{Sym}
\DeclareMathOperator{\Hom}{Hom}

\DeclareMathOperator{\im}{im}

\DeclareMathOperator{\sgn}{sgn}

\DeclareMathOperator{\Bir}{Bir}

\DeclareMathOperator{\Var}{Var}
\DeclareMathOperator{\Vol}{Vol}

\DeclareMathOperator{\MHM}{MHM}
\DeclareMathOperator{\codim}{codim}
\DeclareMathOperator{\tr}{tr}
\DeclareMathOperator{\Bl}{Bl}

\DeclareMathOperator{\spr}{sp}
\DeclareMathOperator{\SB}{SB}
\DeclareMathOperator{\DM}{DM}
\DeclareMathOperator{\Stk}{Stk}
\DeclareMathOperator{\Symm}{Symm}
\DeclareMathOperator{\Hilb}{Hilb}
\DeclareMathOperator{\id}{id}

\DeclareMathOperator{\reg}{reg}

\DeclareMathOperator{\red}{red}
\DeclareMathOperator{\bit}{bit}
\DeclareMathOperator{\Spc}{Spc}
\DeclareMathOperator{\NE}{NE}
\DeclareMathOperator{\cont}{cont}
\DeclareMathOperator{\MHS}{MHS}

\newcommand{\K}{K_0(\Var_k)}
\newcommand{\Kdim}{K_0(\Var_k^{\dim})}
\newcommand{\KdimS}{K_0(\Var_S^{\dim})}
\newcommand{\KsprS}{K_0(\Var_S^{\spr})}
\newcommand{\Kspr}{K_0(\Var_k^{\spr})}
\DeclareMathOperator{\Kstk}{K_0(\Stk_k^{\dim})}
\DeclareMathOperator{\KstkS}{K_0(\Stk_S^{\dim})}
\DeclareMathOperator{\Kdm}{K_0(\DM_k^{\dim})}
\DeclareMathOperator{\KdmS}{K_0(\DM_S^{\dim})}
\newcommand{\KdimX}{K_0(\Var_X^{\dim})}
\newcommand{\KsprX}{K_0(\Var_X^{\spr})}

\makeatletter
\let\c@equation\c@thm
\makeatother
\numberwithin{equation}{section}

\title[Involution on $\Kdim$ and $\mathbb{D}$-Singularities]{Involution on the Graded Grothendieck Ring of Varieties and $\mathbb{D}$-Singularities}
\author{Andrew Burke}
\address{Department of Mathematics, Harvard University, 1 Oxford Street, Cambridge, MA 02138, USA}
\email{aburke@math.harvard.edu}

\begin{document}

\begin{abstract}
We realize a graded variant $\Kdim$ of the Grothendieck ring of varieties as a quadratic extension of the subring $\Kspr$ spanned by classes of smooth and proper varieties. As such, there exists a natural involution $\mathbb{D}$ on $\Kdim$. We show that $\mathbb{D}$ commutes with the symmetric power operations $\Sym^m$ up to zero divisors. Moreover, we study varieties which are smooth up to cut-and-paste relations, which we call $\mathbb{D}$-singular varieties, and we give applications to compactifications of varieties and the irrationality of Kapranov zeta functions.
\end{abstract}

\maketitle

\setcounter{tocdepth}{1}

\section{Introduction}

Let $k$ be a field, assumed to be of characteristic zero. The Grothendieck ring of varieties $K_0(\Var_k)$ is the abelian group generated by classes $[X]$ of varieties $X$ over $k$, modulo the cut-and-paste relations
\begin{align*}
    [X] = [Y] + [X-Y]
\end{align*}
for each closed subvariety $Y$ of $X$. It is a ring, with product $[X] \cdot [Y] = [X \times Y]$.

This ring was first mentioned by Grothendieck in a $1964$ letter to Serre in relation to motives. It features prominently in the theory of motivic integration, introduced by Kontsevich in $1995$ and developed by Denef and Loeser. Here, integrals are valued in a completed localization of $K_0(\Var_k)$.

Despite its simple definition, the Grothendieck ring of varieties $K_0(\Var_k)$ remains largely mysterious. It is a complicated ring with many zero divisors; see \cite{poonendomain}, \cite{kollar conic}, \cite{borisovzerodiv}, \cite{hassettlai}. Nevertheless, these zero divisors seem to witness important geometric information; see \cite{kuznetsovshinder}, \cite{linshinder}.

The Grothendieck ring was connected to birational geometry in \cite{larsenluntsstable} via an isomorphism $K_0(\Var_k)/([\mathbb{A}^1_k]) \cong \mathbb{Z}[\SB_k]$, where $\SB_k$ is the monoid of stable birational equivalence classes. This isomorphism is exploited in \cite{stablebirational} to prove that stable rationality specializes in smooth families. A modified argument in \cite{kontsevichtschinkel} removes the adjective `stable.' The two perspectives are unified in \cite{K^dim} by considering a Grothendieck ring of varieties $\Kdim$, graded by dimension, which can detect birational type.

In this paper, we study $\Kdim$ as an object worthy of independent interest. Explicitly, it is the graded abelian group generated in degree $d$ by classes $[X]_d$ of varieties $X$ of dimension at most $d$, modulo the cut-and-paste relations
\begin{align*}
    [X]_d = [Y]_d + [X-Y]_d,
\end{align*}
for each closed subvariety $Y$ of $X$. It is a graded ring, with $[X]_d \cdot [Y]_e = [X \times Y]_{d+e}$.

There is a graded subring $\Kspr$ of $\Kdim$ spanned by the classes $[X]_n$ of smooth and proper varieties $X$ of dimension $n$. Note that $\Kspr$ is a proper subring, as, for instance, it does not contain the classes $[\mathop{\mathrm{Spec}}k]_1$ or $[\mathbb{A}^1]_1$. Nevertheless, these two classes are solutions of the integral quadratic equation
\begin{align*}
    \big( \hspace{.1cm} x^2 - [\mathbb{P}^1]_1 \cdot x + \big( [ \mathbb{P}^1 \times \mathbb{P}^1 ]_2 - [\mathbb{P}^2]_2 \big) \hspace{.1cm} \big) \in \Kspr[x].
\end{align*}
Our main structural result says that $\Kdim$ is a quadratic extension of $\Kspr$.

\begin{thm}\label{intro str iso}
There is an isomorphism of graded rings
\begin{align*}
    \Kdim \cong \frac{\Kspr \big[ \tau, \mathbb{L} \big] }{ \big( \hspace{.1cm} \tau + \mathbb{L} -  [\mathbb{P}^1]_1, \hspace{.2cm} \tau \cdot \mathbb{L} -  \big( [ \mathbb{P}^1 \times \mathbb{P}^1 ]_2 - [\mathbb{P}^2]_2 \big) \hspace{.1cm} \big)}.
\end{align*}
We view $\tau$ and $\mathbb{L}$ as indeterminates of degree one on the right side; they map to $[\mathop{\mathrm{Spec}}k]_1$ and $[\mathbb{A}^1]_1$, respectively.
\end{thm}
Put more simply, as an abelian group, we have
\begin{align*}
    \Kdim \cong \Kspr \oplus \tau \cdot \Kspr.
\end{align*}
Moreover, there exists an involution
\begin{align*}
    \mathbb{D}: \Kdim \rightarrow \Kdim,
\end{align*}
which fixes $\Kspr$ pointwise and interchanges $\tau = [\mathop{\mathrm{Spec}}k]_1$ and $\mathbb{L} = [\mathbb{A}^1]_1$, as well as projections $\pi_1, \pi_2: \Kdim \rightarrow \Kspr$ according to the above direct sum decomposition.

Though our formalism only works in the graded setting, one finds reflections of it on $K_0(\Var_k)$. For example, Bittner's involution \cite{bittnerrelations} on $K_0(\Var_k)[\mathbb{L}^{-1}]$ is an ungraded manifestation of $\mathbb{D}$ (see Remark \ref{remark bittner} for the precise connection). Our setup is more flexible because $\mathbb{D}$ fixes the classes of smooth and proper varieties, and it does not require the inversion of $\mathbb{L}$, thereby losing birational data. Our formalism also yields simple proofs of the main structural results of \cite{larsenluntsstable} and \cite{zakharevich} on $K_0(\Var_k)$ (see Subsection \ref{motivic measures}).

For the remainder of the introduction, we discuss three main ways we use the involution $\mathbb{D}$ and projections $\pi_1, \pi_2$ defined via Theorem \ref{intro str iso}.

\subsection{Symmetric Powers}
The graded Grothendieck ring of varieties $\Kdim$ is equipped with symmetric power operations
\begin{align*}
    \Sym^m: \Kdim \rightarrow \Kdim, \hspace{.5cm} [X]_d \mapsto [\Sym^m X]_{md},
\end{align*}
for each $m \geq 0$, where $\Sym^m X = X^m / S_m$.

By Proposition \ref{K0 lambda defn}, the classes $\tau$ and $\mathbb{L}$ behave well with respect to $\Sym^m$: $\Sym^m(\tau \alpha) = \tau^m \Sym^m(\alpha)$ and $\Sym^m(\mathbb{L} \alpha) = \mathbb{L}^m \Sym^m(\alpha)$ for $\alpha \in \Kdim$. Motivated by the fact that $\mathbb{D}$ interchanges $\tau$ and $\mathbb{L}$, we conjecture that symmetric powers commute with involution.

\begin{conj}\label{Sym commute D?}
$\Sym^m \circ \mathbb{D} = \mathbb{D} \circ \Sym^m$ for each $m \geq 0$
\end{conj}

 One reduces to checking whether $\mathbb{D}[\Sym^m X] = [\Sym^m X]$ for smooth and proper varieties $X$. This is really a question about the singularities of symmetric powers $\Sym^m X$, of which we ask a more refined version in Conjecture \ref{is sym D-sing?}. It is clearly true for curves and holds for surfaces by the formula of \cite{gottschehilbert}.

We have the following partial answer to Conjecture \ref{Sym commute D?}.

\begin{thm}\label{Sym commute D torsion}
The operations $\Sym^m$ and $\mathbb{D}$ commute, up to zero divisors. More precisely, the image of $\Sym^m \circ \mathbb{D} - \mathbb{D} \circ \Sym^m$ lies in
\begin{align*}
    \ker \bigg( \Kdim \rightarrow \Kdim \big[ (\tau \mathbb{L})^{-1}, (\mathbb{L}^n-\tau^n)^{-1}_{n \geq 1} \big] \bigg).
\end{align*}
\end{thm}

Our proof identifies the localization $\Kdim \big[ \mathbb{L}^{-1}, (\mathbb{L}^n-\tau^n)^{-1}_{n \geq 1} \big]$ with a suitable graded Grothendieck ring of stacks $\Kstk$. Following arguments of Ekedahl \cite{ekedahlstacks}, \cite{ekedahlgroup}, we show the class of $\Sym^m X$ is equal to the class of the stack quotient $[X^m/S_m]$ in $\Kstk$. We conclude by showing the class of any smooth and proper Deligne-Mumford stack is fixed by an extension of the involution $\mathbb{D}$ to $\Kstk[\tau^{-1}]$. 

The generating function for the classes of symmetric powers of a variety $X$ of dimension $n$ is called the Kapranov zeta function of $X$:
\begin{align*}
    Z_X^{\dim}(t) = \sum_{m \geq 0} [\Sym^m X]_{mn} \cdot t^m \in \Kdim[[t]].
\end{align*}
If $k$ is finite, then $Z_X^{\dim}(t)$ specializes under the point counting measure to the zeta function of $X$; this is a rational function by a theorem of Dwork. It is natural to ask whether $Z_X^{\dim}(t)$ is also rational. Dwork's argument shows $Z_X^{\dim}(t)$ is rational when $X$ is a curve, but \cite{larsenluntsrationality} shows when $X$ is a surface that $Z_X^{\dim}(t)$ is rational if and only if $\kappa(X) = -\infty$. We extend one implication to arbitrary dimension.

\begin{thm}\label{intro Kapranov}
Let $X$ be a smooth and projective variety of dimension $>1$ and non-negative Kodaira dimension $\kappa(X) \geq 0$. Then the graded Kapranov zeta function $Z_X^{\dim}(t) \in K_0(\Var_k^{\dim})[[t]]$ of $X$ is (pointwise) irrational. 

In fact, the ungraded Kapranov zeta function $Z_X(t) \in K_0(\Var_k)[[t]]$ of $X$ is also (pointwise) irrational, if we assume a positive answer to Conjecture \ref{Sym commute D?}.
\end{thm}

In fact, a recently posted preprint \cite{shein} independently proves irrationality unconditionally in the ungraded setting.

Our proof of Theorem \ref{Sym commute D torsion} suggests that symmetric power operations may be better behaved in $\Kstk$. It would be interesting to know whether the stacky Kapranov zeta function $Z_X^{\Stk}(t) \in \Kstk[\tau^{-1}][[t]]$ is irrational. Our argument fails in this setting.

\subsection{Gluing Morphisms}
One can define a relative graded Grothendieck group of varieties $K_0(\Var_X^{\dim})$ over a base scheme $X$. Here, the classes are of varieties equipped with a morphism to $X$, which the cut-and-paste relations are required to respect. We use relative dimension to measure graded degree. Our structural results extend immediately to this setting. In fact, the involution $\mathbb{D}_X$ is part of a six-functor formalism on $K_0(\Var_{X}^{\dim})$.

Suppose we have a closed subscheme $Y \subset X$, with complement $U = X - Y$. Then there exists a canonical direct sum decomposition
\begin{align*}
    K_0(\Var_X^{\dim}) = K_0(\Var_U^{\dim}) \oplus K_0(\Var_Y^{\dim}).
\end{align*}
The relationship between the involutions $\mathbb{D}_X$, $\mathbb{D}_U$, and $\mathbb{D}_Y$ is more involved. It is captured by a map
\begin{align*}
    \mathfrak{g}_U^Y: K_0(\Var_U^{\spr}) \rightarrow K_0(\Var_Y^{\spr}),
\end{align*}
which we call the gluing morphism of $Y$ to $U$ via $X$. More specifically, $\mathfrak{g}_U^Y = -i^* \pi_2 j_!$ is defined in terms of the projection $\pi_2$ from earlier.

We explore the geometric data witnessed by gluing morphisms.

\begin{thm}\label{intro glue thm}
Let $U$ be a smooth variety. Choose a smooth compactification $X$ of $U$ whose boundary $D = X - U$ is an snc divisor $D = \cup_{i \in I} D_i$. Then the class
\begin{align*}
    \sum_{\emptyset \neq J \subset I} (-1)^{|J|-1} [D_J \times \mathbb{P}^{|J|-1}] \in \Kspr.
\end{align*}
depends only on $U$, where we denote $D_J = \cap_{j \in J} D_j$ for $J \subset I$.

In particular, if $D$ is smooth, then the birational type and Hodge numbers of $D$ are independent of the choice of compactification.
\end{thm}
This result is likely already known to experts and can be proved more directly by careful application of the weak factorization theorem. Nevertheless, our approach is more systematic and allows for generalizations, as in Section \ref{gluing map}.

\begin{exmp}
If $X = \mathop{\mathrm{Spec}}k[[t]]$ with $U = \mathop{\mathrm{Spec}}k((t))$ and $Y = \mathop{\mathrm{Spec}}k$, the gluing map
\begin{align*}
    \mathfrak{g}_U^Y: K_0(\Var_{k((t))}^{\spr}) \rightarrow \Kspr
\end{align*}
extends to recover the volume morphism $\Vol:  K_0(\Var_{k((t))}^{\dim}) \rightarrow \Kdim$ constructed by hand in \cite{K^dim}. It is the key ingredient in their proof that birational type specializes in smooth families, and it is exploited in \cite{tropical} and \cite{moe} to establish new examples of irrational hypersurfaces.
\end{exmp}

\subsection{$\mathbb{D}$-Singularities} 
We study the following condition, which captures what it means for a variety to be smooth from the perspective of the graded Grothendieck group. Its definition, like that of the gluing morphism, relies on Theorem \ref{intro str iso}.

\begin{defn}
We say a variety $X$ has $\mathbb{D}$-singularities if the identity morphism $\mathbbm{1}_X = [X \xrightarrow[]{id} X]_0$ of $\KdimX$ lies in the subgroup $\KsprX$. Equivalently, $\pi_2(\mathbbm{1}_X) = 0$.
\end{defn}

This is a local condition. If $X$ is proper and has $\mathbb{D}$-singularities, then $\mathbb{D}[X] = [X]$, an identity which automatically holds if $X$ is also smooth.

Concretely, a variety $X$ has $\mathbb{D}$-singularities if it can be constructed by cutting into pieces and then gluing back together a $\mathbb{Z}$-linear combination of smooth varieties that are proper over $X$ and of the same dimension as $X$.

Our definition of $\mathbb{D}$-singularities refines the more ad hoc notions of $\mathbb{L}$-rational singularities \cite{stablebirational} and B-rational singularities \cite{kontsevichtschinkel} found in the literature (see Examples \ref{L exmp} and \ref{B exmp}). Moreover, it can be computed effectively, provided we understand a log resolution of $X$. We present a few examples here to help with intuition. More can be found in Section \ref{D sing section}.

\begin{exmp}
A curve $X$ has $\mathbb{D}$-singularities if and only if it is unibranched at each point.
\end{exmp}

\begin{exmp}
Let $X$ be a cone over a smooth and proper variety $Y$ of dimension $n$. Then $X$ has $\mathbb{D}$-singularities if and only if
\begin{align*}
    [Y] = [\mathbb{P}^n] \hspace{.3cm} \text{in} \hspace{.3cm} \Kdim.
\end{align*}
For example, a cone over a Veronese embedding of $\mathbb{P}^n$ or over an odd-dimensional complex quadric has $\mathbb{D}$-singularities.
\end{exmp}

In each of these cases, if $X$ has $\mathbb{D}$-singularities, then it is a rational homology manifold, i.e. the rational homology of the link at each singular point is equal to that of a sphere. More generally, Proposition \ref{D-sing implies rhm} shows that if $X$ is a local complete intersection with isolated singularities and $\mathbb{D}$-singularities, then $X$ is a rational homology manifold. Nevertheless, there exist $\mathbb{D}$-singularities which are not rational homology manifolds, as in Example \ref{rhm example}.

Adapting arguments of Esser-Scavia \cite{esserquotient} and Bittner \cite{bittnertoric}, we prove:

\begin{thm}\label{intro abelian quot}
Let $G$ be a finite abelian group acting on a smooth quasi-projective variety $X$. Then the quotient variety $X/G$ has $\mathbb{D}$-singularities. 
\end{thm}

Not all quotient singularities are $\mathbb{D}$-singularities; this is related to the failure of Noether's question. Conjecture \ref{is sym D-sing?} is a refined version of Conjecture \ref{Sym commute D?} which asks, nevertheless, whether the symmetric powers $\Sym^m X$ of a smooth variety $X$ have $\mathbb{D}$-singularities. We prove in Theorem \ref{sym d-sing} that this is true, up to potential zero divisors. A recent preprint \cite{shein} provides further evidence by proving $\Sym^m X$ has $\mathbb{L}$-rational singularities.

\subsection*{Acknowledgements}
I would like to sincerely thank my advisor, Mihnea Popa, for his constant support during the preparation of this paper. I am also grateful to Brendan Hassett, Mircea Musta{\c{t}}{\u{a}}, Evgeny Shinder, and Anh Duc Vo for many helpful conversations and comments.

\section{Graded Grothendieck groups}

\subsection{Preliminaries}
Let $S$ be a base scheme. 

Here and throughout, we require that our base schemes be noetherian and excellent of characteristic zero. We are primarily interested in the case when $S$ is a quasi-projective variety or the spectrum of a discrete valuation ring over a field $k$ of characteristic zero. We impose such conditions on $S$ to access results on resolution of singularities, the weak factorization theorem, and a well-behaved notion of relative dimension. 

Namely, let $f: X \rightarrow S$ be a morphism of finite type and assume that $X$ is irreducible. Then we set 
\begin{align*}
    \dim_S X = \tr \deg ( k(X), k(V) ) - \codim(V,S),
\end{align*}
where $V$ is the closure of $f(X)$ in $S$. In particular, if $S$ is a variety over a field $k$, we have $\dim_S X = \dim X - \dim S$. We will use the following basic properties of relative dimension without further comment.

\begin{lem}\label{dimension properties}
\begin{enumerate}
    \item If $U$ is an open subscheme of $X$, then $\dim_S U = \dim_S X$.
    \item If $Y \subset X$ is a closed irreducible subscheme, then 
        \begin{align*}
            \dim_S X = \dim_S Y + \codim(Y, X).
        \end{align*}
    \item If $Y \rightarrow X$ is a dominant morphism of irreducible schemes over $S$, then
        \begin{align*}
            \dim_S Y = \dim_S X + \tr \deg(k(Y), k(X)).
        \end{align*}
    \item If $X \xrightarrow[]{f} S \rightarrow T$, then $\dim_T X = \dim_S X + \dim_T S$.
    \item If $X \rightarrow S$ is flat of relative dimension $d$, then $\dim_S X = d$.
\end{enumerate}
\end{lem}

\begin{proof}
The first three properties hold by \cite[Lemma $20.1$]{fulton2013intersection}.

For $(4)$, denoting by $V$ the closure of the image $f(X)$ in $S$, we have $\dim_S X = \dim_S V + \tr \deg(k(X), k(V))$ and $\dim_T X = \dim_T V + \tr \deg(k(X), k(V))$ by $(3)$. We therefore reduce to the case where $X = V$ is a closed subscheme of $S$. The result is then immediate from $(2)$. 

Property $(5)$ is clear.
\end{proof}

If $X$ is possibly reducible, let $\dim_S X = \sup \dim_S X_i$ be the maximum relative dimension, taken over the irreducible components $X_i$ of $X$.

\begin{defn}\label{graded K0S}
The graded Grothendieck group $\KdimS$ of $S$-varieties is the $\mathbb{Z}$-graded abelian group with:

\begin{itemize}
    \item generators in degree $d$ the isomorphism classes $[X]_d$ of $S$-schemes $X$ of finite type, with $d \geq \dim_S X$, and
    \item relations $[X]_d = [Y]_d + [X-Y]_d$ for $X$ an $S$-scheme of finite type with $d \geq \dim_S X$ and $Y$ a closed subscheme of $X$.
\end{itemize} 
\end{defn}

We will call a relation as above a cut-and-paste or gluing relation. We sometimes write $[X]_d$ as $[X \xrightarrow[]{f} S]_d$ to emphasize the morphism $f$. 

Denote by $\mathbbm{1}_S =  [S \xrightarrow[]{id} S]_0 \in \KdimS$ the identity morphism of $S$. We will also write
\begin{align*}
    \tau = [\mathop{\mathrm{Spec}}k]_1 \hspace{.4cm} \mathrm{and} \hspace{.4cm} \mathbb{L} = [\mathbb{A}_k^1]_1 \hspace{.4cm} \mathrm{in} \hspace{.2cm} \Kdim.
\end{align*}

\subsection{Functoriality}
Consider a morphism $\phi: S \rightarrow T$ of finite type. There is a pushforward homomorphism
\begin{align*}
    \phi_!: \KdimS \rightarrow K_0(\Var_T^{\dim}), \hspace{.5cm} [X \xrightarrow[]{f} S]_d \mapsto [X \xrightarrow[]{\phi \circ f}T]_{d+ \dim_T S}
\end{align*}
of graded degree $\dim_T S$.

We also have a pullback
\begin{align*}
    \phi^*: K_0(\Var_T^{\dim}) \rightarrow K_0(\Var_S^{\dim}), \hspace{.5cm} [X \rightarrow T]_d \mapsto [X \times_T S \rightarrow S]_{d + \deg \phi^*},
\end{align*}
defined if $\phi$ is flat or a closed embedding. We set $\deg \phi^* = 0$ if $\phi$ is flat and $\deg \phi^* = k$ if $\phi$ is a closed embedding of codimension $k$.

If $S$ and $T$ are of finite type over a field $k$, there is an exterior product
\begin{align*}
    \boxtimes: \KdimS \times K_0(\Var_T^{\dim}) \rightarrow K_0(\Var_{S \times T}^{\dim})
\end{align*}
defined by 
\begin{align*}
    \big( [X \xrightarrow[]{f} S]_d, [Y \xrightarrow[]{g} T]_e \big) \mapsto [X \times Y \xrightarrow[]{f \times g} S \times T]_{d+e}.
\end{align*}

Thus, $\Kdim$ naturally inherits the structure of a graded ring, and $\KdimS$ is a graded $\Kdim$-module via
\begin{align*}
    [Y]_d \cdot [X \rightarrow S]_e = [X \times Y \rightarrow S]_{d+e}.
\end{align*}
Moreover, in this situation, $\phi_!$ and $\phi^*$ are both $\Kdim$-module homomorphisms.

Even if $S$ is not of finite type over a field, we view $\KdimS$ as a graded  module over $\mathbb{Z}[\tau, \mathbb{L}]$, where $\tau$ acts by $\tau \cdot [X]_d = [X]_{d+1}$ and $\mathbb{L}$ acts by $\mathbb{L} \cdot [X]_d = [X \times_S \mathbb{A}^1_S]_{d+1}$.

\subsection{Involution}

We will construct an involution $\mathbb{D}$ on $\KdimS$ by carefully studying an alternative Bittner-style presentation.

\begin{thm}\label{bittner presentation}
The graded Grothendieck group $\KdimS$ of $S$-varieties admits a presentation with:
\begin{itemize}
    \item generators in degree $d$ the isomorphism classes $[X]_d$ of regular, connected $S$-schemes $X$, proper and of finite type over $S$, with $d \geq \dim_S X$, and
    \item relations
        \begin{align*}
            [\Bl_Y X]_d - [E]_d = [X]_d - [Y]_d,
        \end{align*}
    for $\Bl_Y X$ the blow-up of a regular, connected $S$-scheme $X$, proper and of finite type over $S$, with $d \geq \dim_S X$ along a regular and connected closed subscheme $Y$, and with $E$ the exceptional divisor.
\end{itemize}
\end{thm}

\begin{proof}
Our argument closely mirrors the analogous ungraded argument in \cite[Theorems $3.1$ and $5.1$]{bittnerrelations}. We provide it for the convenience of the reader and for future reference.

Denote by $K_0(\Var_S^{\bit})$ the graded abelian group with presentation as in the statement of the Theorem. The natural map $K_0(\Var_S^{\bit}) \rightarrow \KdimS$ factors as a composition

\[
\begin{tikzcd}
K_0(\Var_S^{\bit}) \arrow[rd, "\Psi"'] \arrow[rr] & & \KdimS \\
& K_0(\Var_S^{\reg}) \arrow[ru, "\Phi"'] &
\end{tikzcd}
\]
where $K_0(\Var_S^{\reg})$ is the graded Grothendieck group of regular $S$-schemes, defined exactly as in Definition \ref{graded K0S}, except that all schemes are required to be regular. 

\textit{Step 1:} We show $\Phi: K_0(\Var_S^{\reg}) \rightarrow \KdimS$ is an isomorphism of graded abelian groups by constructing an inverse $\Phi^{-1}$.

Let $f:X \rightarrow S$ be a morphism of finite type and let $d \geq \dim_S X$. If necessary, replace $X$ with its reduction; this is harmless, as $[X]_d = [X_{\red}]_d$. Choose a finite stratification
\begin{align*}
    X = \coprod_{i \in I} X_i
\end{align*}
of $X$ by regular subschemes such that each of their closures $\overline{X_i}$ is a union of strata. This is possible, as $S$ being a J$2$ scheme implies that the regular locus of $X$ is open. Then set $\Phi^{-1}([X]_d) = \sum_{i \in I} [X_i]_d$. 

We need to show that this expression is well-defined. If $X$ is itself smooth, we have $[X]_d = \sum_{i \in I} [X_i]_d \in K_0(\Var_S^{\reg})$, as follows by induction on the size of the index set $I$. More generally, consider two stratifications $\coprod_{i \in I} X_i$ and $\coprod_{j\in J} X_j$ of arbitrary $X$. We can find a common refinement $\coprod_{k \in K} X_k$. Then, by the above, we deduce $[X_i]_d = \sum_{X_k \subset X_i} [X_k]_d$ for each $i \in I$. Taking the sum over $i \in I$, we obtain $\sum_{i \in I} [X_i]_d = \sum_{k \in K} [X_k]_d$. Combining this with the same argument involving the index set $J$, we deduce
\begin{align*}
    \sum_{i \in I} [X_i]_d = \sum_{k \in K} [X_k]_d = \sum_{j \in J} [X_j]_d.
\end{align*}

If $Y \subset X$ is a closed subscheme, we can find a stratification of $X$ such that $Y$ is a union of strata. This implies that $\Phi^{-1}([X]_d) = \Phi^{-1}([X-Y]_d) + \Phi^{-1}([Y]_d)$, so $\Phi^{-1}$ defines a group homomorphism $\Phi^{-1}: K_0(\Var_S^{\reg}) \rightarrow \KdimS$. It is inverse to $\Phi$ by construction.

\textit{Step 2:} We show $\Psi: K_0(\Var_S^{\bit}) \rightarrow K_0(\Var_S^{\reg})$ is an isomorphism of graded abelian groups by constructing an inverse $\Psi^{-1}$.

Let $f:X \rightarrow S$ be a morphism of finite type, with $X$ regular. By Nagata compactification, there exists an open embedding $j: X \rightarrow \overline{X}$ and proper morphism $\overline{f}: \overline{X} \rightarrow S$ such that $f = \overline{f} \circ j$. By resolution of singularities as in \cite{temkin2008desingularization}, we may assume $\overline{X}$ is regular and $D = \overline{X} - X$ is an snc divisor. Write $D = \cup_{i \in I} D_i$ and set $D_J = \cap_{j\in J} D_j$ for $J \subset I$. Then we define
\begin{align*}
    \Psi^{-1}([X]_d) = \sum_{J \subset I} (-1)^{|J|} [ D_J ]_d
\end{align*}
for $d \geq \dim_S X$.

We need to show that the expression for $\Psi^{-1}([X]_d)$ is independent of choices. Consider two such regular compactifications $X \rightarrow \overline{X}$ and $X \rightarrow \overline{X}'$. By the weak factorization theorem, \cite{weakfactorization}, Theorem $1.2.1$, they differ by a sequence of blow-up and blow-downs of regular subvarieties with normal crossings. Thus we reduce to the case of $\overline{X}' = \Bl_Z \overline{X}$, where $Z$ is a regular connected subscheme having normal crossings with $D$.

Write $D' = \overline{X}' - X$ as $\cup_{i \in I \cup \{ 0 \}} D_i'$, where $D_0'$ is the exceptional divisor over $Z$ and $D_i'$ is the strict transform of $D_i$ for $i \in I$. In fact, for $J \subset I$, we see $D_J' = \cap_{j \in J} D_j'$ is the blow-up of $D_J$ along $Z_J = Z \cap D_J$; the exceptional divisor is $E_J = D_J' \cap D_0'$. We deduce
\begin{align*}
    \sum_{J \subset I \cup \{ 0\} } (-1)^{|J|} [D_J']_d =& \sum_{J \subset I} (-1)^{|J|} \big( [D_J]_d - [Z_J]_d + [E_J]_d \big) + \sum_{J \subset I} (-1)^{|J|+1} [E_J]_d \\
    =& \sum_{J \subset I } (-1)^{|J|} [D_J]_d - \sum_{J \subset I} (-1)^{|J|} [Z_J]_d.
\end{align*}
Thus we need to show $\sum_{J \subset I} (-1)^{|J|} [Z_J]_d = 0$. As $Z$ has normal crossings with $D$, certainly $Z \subset D_{i_0}$ for some $i_0 \in I$. But then $Z_{J} = Z_{J \cup \{ i_0 \}}$ for each $J \subset I - \{i_0\}$. The vanishing follows by splitting the alternating sum into those terms containing $i_0$ and those not containing $i_0$.

Let $Y \subset X$ be a regular closed subscheme of $X$. We need to show 
\begin{align*}
    \Psi^{-1}([X]_d) = \Psi^{-1}([X-Y]_d) + \Psi^{-1}([Y]_d).
\end{align*}

Choose an open embedding $X \rightarrow \overline{X}$ where $\overline{X}$ is regular and $D = \overline{X} - X$ is an snc divisor such that the closure $\overline{Y}$ of $Y$ is smooth and has normal crossings with $D$. We can achieve this, for example, by picking a regular compactification of $X$ with snc boundary, and then replacing it with an embedded resolution of the closure of $Y$, compatible with the boundary. Set $D = \cup_{i \in I} D_i$ and denote $D_J = \cap_{j \in J} D_j$ and $Y_J = Y \cap D_J$ for $J \subset I$.

Consider the blow-up $\Bl_{\overline{Y}} \overline{X}$ of $\overline{X}$ along $\overline{Y}$. It is a regular compactification of $X-Y$ with boundary $\cup_{i \in I \cup \{0\}} D_i'$; here $D_i'$ is the strict transform of $D_i$ for $i \in I$ and $D_0$ is the exceptional divisor. Viewing $D_J' = \cap_{j \in J} D_j'$ as the blow-up of $D_J$ with center $Y_J$, we compute
\begin{align*}
    \Psi^{-1}([X-Y]_d) =& \sum_{J \subset I \cup \{0\}} (-1)^{|J|} [D_J']_d \\
    =& \sum_{J \subset I} (-1)^{|J|} \big( [D_J]_d - [Y_J]_d + [E_J]_d \big) - \sum_{J \subset I} (-1)^{|J|+1} [E_J]_d \\
    =& \Psi^{-1}([X]_d) - \Psi^{-1}([Y]_d).
\end{align*}

Therefore $\Psi^{-1}$ defines a group homomorphism $K_0(\Var_S^{\reg}) \rightarrow K_0(\Var_S^{\bit})$. It is clearly inverse to $\Psi$.

Putting Steps $1$ and $2$ together, we conclude that the composition $\Phi \circ \Psi: K_0(\Var_S^{\bit}) \rightarrow \KdimS$ is a graded group isomorphism.
\end{proof}

\begin{rmk}
Observe that the geometric inputs to this Bittner presentation are resolution of singularities, and the weak factorization theorem. We will later need a Bittner presentation for an appropriate Grothendieck group of Deligne-Mumford stacks. Essentially the same proof works in this context, as the geometric ingredients remain valid.
\end{rmk}

We now reinterpret the Bittner presentation, viewing $\KdimS$ as a module over $\mathbb{Z}[\tau, \mathbb{L}]$, where $\tau$ and $\mathbb{L}$ play symmetric roles.

\begin{cor}\label{bit module}
The graded Grothendieck group $\KdimS$ of $S$-varieties admits an alternative presentation as a $\mathbb{Z}[\tau, \mathbb{L}]$-module with:

\begin{itemize}
    \item generators in degree $d$ the isomorphism classes $[X]$ of regular, connected $S$-schemes $X$, which are proper, of finite type and of relative dimension $d$ over $S$, and
    \item relations
        \begin{align}\tag{1}
            [\Bl_Y X] = [X] + \tau \mathbb{L} \cdot (\tau^{k-1} + \ldots + \mathbb{L}^{k-1}) \cdot [Y]
        \end{align}
        for $\Bl_Y X$ the blow-up of a regular, connected $S$-scheme $X$, proper and of finite type over $S$ along a regular and connected subscheme $Y$ of codimension $k+1$, and with $E$ the exceptional divisor, and 
        \begin{align}\tag{2}
            [E] = (\tau^{k}+ \tau^{k-1} \mathbb{L} +  \ldots + \mathbb{L}^{k}) \cdot [Y]
        \end{align} 
        for $E \rightarrow Y$ a Zariski locally-trivial $\mathbb{P}^k$-fibration over a regular, connected $S$-scheme $Y$, proper and of finite type over $S$.
        
\end{itemize}
\end{cor}

\begin{proof}
Temporarily denote by $K_0(\Var_S^{\tau, \mathbb{L}})$ the $\mathbb{Z}[\tau, \mathbb{L}]$-module defined by the above presentation. We make use of the Bittner presentation on $\KdimS$. Consider $\Theta: \KdimS \rightarrow K_0(\Var_S^{\tau, \mathbb{L}})$ defined on generators by $[X]_d \mapsto [X] \cdot \tau^{d - \dim_S X}$. This is a well-defined group homomorphism, for
\begin{align*}
    \Theta([\Bl_Y X]_d) &- \Theta([E]_d) - \Theta([X]_d) + \Theta([Y]_d) \\
    &= \big( [\Bl_Y X] - \tau \cdot [E] -[X] + \tau^{k+1} \cdot [Y]  \big) \cdot \tau^{d-\dim_S X} \\
    &= \big( [\Bl_Y X] - \tau \cdot (\tau^{k}+ \ldots + \mathbb{L}^{k})[Y] -[X] + \tau^{k+1} \cdot [Y]  \big) \cdot \tau^{d-\dim_S X} \\
    &= \big( [\Bl_Y X] -[X] - \tau \mathbb{L} \cdot ( \tau^{k-1} + \ldots + \mathbb{L}^{k-1}) \cdot [Y]  \big) \cdot \tau^{d-\dim_S X} \\
    &= 0,
\end{align*}
if $\Bl_Y X \rightarrow X$ is the blow-up along a regular subscheme $Y$ and $d \geq \dim_S X$. It is moreover, a $\mathbb{Z}[\tau, \mathbb{L}]$-module homomorphism.

Define an inverse $\Theta^{-1}:  K_0(\Var_k^{\tau, \mathbb{L}}) \rightarrow \KdimS$ on generators by $[X] \mapsto [X]_{\dim_S X}$. To see that it is well-defined, suppose $E \rightarrow Y$ is a Zariski locally-trivial $\mathbb{P}^k$-fibration. Then
\begin{align*}
    \Theta^{-1}([E]) = [E]_{\dim_S E} = (\tau^k + \ldots + \mathbb{L}^k) \cdot 
 [Y]_{\dim_S Y} = \Phi^{-1}((\tau^k + \ldots + \mathbb{L}^k) \cdot [Y]).
\end{align*}

Moreover, if $\Bl_Y X \rightarrow X$ is the blow-up along a regular subscheme $Y$, we compute
\begin{align*}
    \Theta^{-1}&([\Bl_Y X]) - \Theta^{-1}([X]) - \Theta^{-1}(\tau \mathbb{L} \cdot (\tau^{k-1} + \ldots + \mathbb{L}^{k-1}) \cdot [Y]) \\
    =& [\Bl_Y X]_d - [X]_d - \big( \tau \cdot (\tau^k + \ldots + \mathbb{L}^k) - \tau^k \big) \cdot [Y]_{d-k-1} \\
    =& [\Bl_Y X]_d - [X]_d - [E]_{d} + [Y]_{d} \\
    =& 0.
\end{align*}
The result follows.
\end{proof}

From now on, we will write $[X]$ as shorthand for $[X]_{\dim_S X}$ in $\KdimS$.

Denote by $\KsprS$ the graded subgroup of $\KdimS$ generated by the classes $[X]$ of regular connected schemes $X$, proper over $S$. It is a $\mathbb{Z}[\tau+\mathbb{L}, \tau \mathbb{L}]$-module.

Note that all generators and relations in Corollary \ref{bit module} take place inside of $\KsprS$. It precisely gives a presentation for $\KsprS$ as a $\mathbb{Z}[\tau+\mathbb{L}, \tau \mathbb{L}]$-module. From this, we deduce the following structural result.

\begin{thm}\label{str iso}[cf. Theorem \ref{intro str iso}]
There exists a natural isomorphism
\begin{align*}
    \KsprS \otimes_{\mathbb{Z}[\tau + \mathbb{L}, \tau \mathbb{L}]} \mathbb{Z}[\tau, \mathbb{L}] \rightarrow \KdimS
\end{align*}
of $\mathbb{Z}[\tau, \mathbb{L}]$-modules.
\end{thm}

In particular, $\KdimS$ is a quadratic extension of $\KsprS$. As such, there is a Galois involution.

\begin{defn}
Let $\mathbb{D}: \KdimS \rightarrow \KdimS$ be the graded $\mathbb{Z}[\tau+ \mathbb{L}, \tau \mathbb{L}]$-module involution $\mathbb{D}(\alpha \otimes a(\tau, \mathbb{L}) ) = \alpha \otimes a(\mathbb{L}, \tau)$ in the notation of Theorem \ref{str iso}.

Moreover, every $\alpha \in \KdimS$ can be writted uniquely as $\alpha = \pi_1(\alpha) + \tau \cdot \pi_2(\alpha)$, where $\pi_1(\alpha), \pi_2(\alpha) \in \KsprS$. Denote by
\begin{align*}
    \pi_1: \KdimS \rightarrow \KsprS, \hspace{.5cm} \pi_2: \KdimS \rightarrow \KsprS
\end{align*}
the projection maps; they are of graded degree $0$ and $-1$ respectively.
\end{defn}

The graded involution $\mathbb{D}$ fixes classes of the form $[X]$, where $X$ is regular and proper over $S$, and it interchanges the actions of $\tau$ and $\mathbb{L}$. In fact, $\mathbb{D}$ is characterized by these two properties.

\begin{exmp}\label{P^k example}
For future reference, we record the values $\pi_1(\tau^{k+1}) = -\tau \mathbb{L} \cdot [\mathbb{P}^{k-1}]$ and $\pi_2(\tau^{k+1}) = [\mathbb{P}^{k}]$. Indeed, it suffices to check that
\begin{align*}
    \tau^{k+1} = -\tau \mathbb{L} \cdot [\mathbb{P}^{k-1}] + \tau \cdot [\mathbb{P}^{k}],
\end{align*}
where we view this identity as taking place formally in $\mathbb{Z}[\tau, \mathbb{L}]$, with $[\mathbb{P}^k]$ being shorthand for $\tau^k + \tau^{k-1}\mathbb{L} + \ldots + \mathbb{L}^k$.
\end{exmp}

Let us return to the setting of a morphism $\phi: S \rightarrow T$ of finite type. We have induced maps $\phi_!$ and $\phi^*$. Define $\phi_* = \mathbb{D} \circ \phi_! \circ \mathbb{D}$ and $\phi^! = \mathbb{D} \circ \phi^* \circ \mathbb{D}$.

\begin{lem}\label{proper}
Suppose $\phi:S \rightarrow T$ is proper. Then $\phi_!(\KsprS) \subset K_0(\Var_T^{\spr})$. Hence $\phi_! = \phi_*$
\end{lem}

\begin{proof}
If $X$ is regular and proper over $S$, then $\phi_![X \xrightarrow[]{f} S] = [X \xrightarrow[]{\phi \circ f} T]$, with $X$ proper over $T$, as $\phi \circ f$ is a composition of proper morphisms.
\end{proof}

\begin{lem}\label{smooth}
Let $\phi:S \rightarrow T$ be a smooth morphism of relative dimension $d$. Then $\phi^*(K_0(\Var_T^{\spr})) \subset \KsprS$. Hence $\phi^! = \phi^*$
\end{lem}

\begin{proof}
It suffices to show that $\phi^*[X \xrightarrow[]{f} T] = [X \times_T S \xrightarrow[]{g} S]$ lies in $\KsprS$ if $X$ is a regular scheme, proper over $T$. Consider the following pullback diagram.
\centerline{
\xymatrix{
 X \times_T S \ar[r]^\psi \ar[d]_g & X \ar[d]^f \\
 S \ar[r]_\phi & T \\
 } 
}
The morphism $g$ is proper, as the pullback of the proper morphism $f$. Moreover, since $\psi$, as the pullback of the smooth morphism $\phi$, is smooth and $X$ is regular, we find that $X \times_T S$ is a regular scheme. Thus, indeed, $[X \times_T S \xrightarrow[]{g} S] \in \KdimS$.
\end{proof}

\subsection{Graded Motivic Measures}\label{motivic measures}
It can be difficult to deduce geometric results by working directly with Grothendieck groups of varieties because they are so large. Rather, one probes them by motivic measures; by this we mean homomorphisms $K_0(\Var_S^{\dim}) \rightarrow A$. We might ask that the map preserve other structures on $\Kdim$, for example that it be graded, commute with mutliplication, with involution, or with $\lambda$-structures (see Proposition \ref{K0 lambda defn}). Let us introduce some relevant motivic measures.

\textbf{Ungraded Grothendieck group:} The quotient $\KdimS / (\tau - 1) \KdimS$ recovers the usual ungraded Grothendieck group of varieties $K_0(\Var_S)$. This is the abelian group generated by isomorphism classes $[X]$ of $S$-schemes $X$ of finite type, modulo the relations $[X] = [Y]+[X-Y]$ for every closed subscheme $Y$ of $X$.

Pre-composing with reduction modulo $\tau - 1$, we see all motivic measures on $K_0(\Var_S)$ are also motivic measures on $\KdimS$. Over a finite field $k$, for example, there is the point count measure
\begin{align*}
    K_0(\Var_k) \rightarrow \mathbb{Z}, \hspace{.5cm} [X] \mapsto |X(k)|.
\end{align*}

We also have the Hodge-Deligne polynomial $H: K_0(\Var_\mathbb{C}) \rightarrow \mathbb{Z}[u,v]$, characterized by the relation
\begin{align*}
    H([X]) = H_X(u,v) = \sum_{p,q \geq 0} (-1)^{p+q} h^{p,q}(X) \cdot u^p v^q
\end{align*}
for every smooth and projective complex variety $X$. 

\begin{rmk}\label{remark bittner}
Bittner \cite[Definition $6.3$]{bittnerrelations} defines an involution $\mathbb{D}_{\mathrm{bit}}$ on $K_0(\Var_S)[\mathbb{L}^{-1}]$, characterized by the property $\mathbb{D}_{\mathrm{bit}}([X]) = [X]/\mathbb{L}^{\dim X}$ when $X$ is smooth and proper over $S$. Let us explain the precise connection between $\mathbb{D}$ and $\mathbb{D}_\mathrm{bit}$.

Our involution $\mathbb{D}$ extends to the degree zero part $K_0(\Var_S^{\dim})[(\tau \mathbb{L})^{-1}]_0$ of the localization of $K_0(\Var_S^{\dim})$ at $\tau \mathbb{L}$. There is an isomorphism of modules
\begin{align*}
    K_0(\Var_S)[\mathbb{L}^{-1}] \cong K_0(\Var_S^{\dim})[(\tau \mathbb{L})^{-1}]_0, \hspace{.5cm} [X] \mapsto \frac{[X]}{\tau^{\dim X}}.
\end{align*}
Under this isomorphism, one observes $\mathbb{D}$ coincides with $\mathbb{D}_{\mathrm{bit}}$.
\end{rmk}

\textbf{Birational equivalence classes:} Let $k$ be a field of characteristic zero and let $\mathbb{Z}[\Bir_k]$ (resp. $\mathbb{Z}[\SB_k]$) be the monoid ring of the monoid of (resp. stable) birational equivalence classes over $k$. There is a quotient map $\Kdim \rightarrow \mathbb{Z}[\Bir_k]$ which sends $[X]_d$ to the birational equivalence class of $X$ if $d=\dim X$ and to zero otherwise. It is easily seen to be surjective with kernel generated by $\tau$.

Let us now pause to give a conceptual proof of a theorem of Larsen-Lunts \cite{larsenluntsstable}.

\begin{prop}
$K_0(\Var_k)/(\mathbb{L}) \cong \mathbb{Z}[\SB_k]$
\end{prop}

The proof of Larsen-Lunts involves careful manipulation of the weak factorization theorem and resolution of singularities. Our argument hides such an analysis behind the existence of the involution $\mathbb{D}$ (which still relies on the Bittner presentation).

\begin{proof}
We have the following chain of isomorphisms:
\begin{align*}
    K_0(\Var_k)/(\mathbb{L}) \cong& \Kdim/(\tau-1, \mathbb{L}) \\ \cong& \Kdim/(\mathbb{L}-1,\tau) \\ \cong& \mathbb{Z}[\Bir_k]/([\mathbb{P}^1]-[\mathop{\mathrm{Spec}}k]) \\ \cong& \mathbb{Z}[\SB_k],
\end{align*}
where the second is deduced by pre-composing with the involution $\mathbb{D}$.
\end{proof}

The following simple lemma also has important connections to the literature.

\begin{lem}\label{simple lemma}
Suppose $\alpha \in \Kdim$ satisfies $\tau \alpha = 0$. Then $\alpha = \mathbb{L} \beta$ for a unique class $\beta \in \Kspr$.
\end{lem}
\begin{proof}
We may uniquely write $\alpha = \gamma + \mathbb{L} \beta$, where $\beta, \gamma \in \Kspr$. Then
\begin{align*}
    0 = \pi_2(\tau \alpha) = \pi_2(\tau \gamma + \tau \mathbb{L} \beta) = \gamma.
\end{align*}
\end{proof}

\begin{rmk}
In \cite{linshinder}, the authors study a map of the form
\begin{align*}
    \tilde{c}: \Bir(\mathbb{P}^n) \rightarrow \ker \big( \tau: \Kdim^{n-1} \rightarrow \Kdim^{n} \big)
\end{align*} 
to construct new elements of the Cremona group. Lemma \ref{simple lemma} shows this morphism factors through $\im( \mathbb{L}: \Kspr^{n-2} \rightarrow \Kdim^{n-1})$, a fact evident in all their examples.
\end{rmk}

\begin{rmk}
We also recover \cite[Theorems $C$ and $D$]{zakharevich}. Indeed, suppose $0 \neq \alpha \in K_0(\Var_k)$ satisfies $\mathbb{L} \alpha = 0$. We may lift $\alpha$ to an element $\tilde{\alpha} \in \Kdim$ with $\mathbb{L} \tilde{\alpha} = 0$. But then (the dual of) Lemma \ref{simple lemma} says we can write $\tilde{\alpha} = \tau ( [X]_n - [Y]_n )$, with $X$ and $Y$ smooth and proper of dimension $n$. Thus
\begin{align*}
    0 \neq \mathbb{D}(\tilde{\alpha}) = [X \times \mathbb{A}^1]_{n+1} - [Y \times \mathbb{A}^1]_{n+1}
\end{align*}
implies $X \times \mathbb{A}^1$ and $Y \times \mathbb{A}^1$ are not piecewise isomorphic. Nevertheless, forgetting the grading, we see $\alpha$ is of the form $[X] - [Y]$ with $[X \times \mathbb{A}^1] = [Y \times \mathbb{A}^1]$ in $\K$.
\end{rmk}

\textbf{Mixed Hodge modules:}\label{mhm subsection} Let $S$ be a complex variety. The Hodge-Deligne polynomial generalizes to a map
\begin{align*}
    \chi_{S}^c: \KdimS \rightarrow K_0(\MHM(S)), \hspace{.5cm} \chi^c_{S}([X \xrightarrow[]{f} S]_d) = [f_! \mathbb{Q}_X^{H}],
\end{align*}
called the graded motivic Hodge-Grothendieck characteristic. Here, $K_0(\MHM(S))$ is the Grothendieck group of the abelian category of mixed Hodge modules on $S$, and $\mathbb{Q}_X^{H} = (a_X)^*\mathbb{Q}^H$ is the constant Hodge module on $X$. To see that $\chi_S^c$ respects cut-and-paste relations, one utilizes the exact triangle $j_!j^* \rightarrow \id \rightarrow i_* i^* \xrightarrow[]{+1}$ for a closed embedding $i:Y \rightarrow X$.

In fact, the derived category, and hence also the Grothendieck group, of mixed Hodge modules admits a six functor formalism. It is clear that if $\phi:S \rightarrow T$ is a morphism of varieties, then we have the commutativity relations $\chi^c_{T} \phi_! = \phi_! \chi^c_S$ and $\chi^c_S \phi^* = \phi^* \chi^c_T$. The respective dualities also commute up to a Tate twist: $\chi_S^c \circ \mathbb{D} = (\mathbb{D} \circ \chi^c_S)(d+\dim S)$ on degree $d$ elements of $\KdimS$.

\subsection{Spreading Out}
Let $S$ be a base scheme and let $\eta \in S$ be a point with local ring $A = \mathscr{O}_{\eta, S}$

Consider the inverse system $\big(\mathop{\mathrm{Spec}}A_j \rightarrow \mathop{\mathrm{Spec}}A_i \big)_{j \geq i \in I}$ of open affine subschemes $\mathop{\mathrm{Spec}}A_i$ of $S$ containing the point $\eta$. The inverse limit of this system is $\mathop{\mathrm{Spec}}A$. Alternatively, $A$ is the direct limit of the corresponding direct system $\big( A_i \rightarrow A_j \big)_{i \leq j \in I}$ of noetherian rings.

Taking the relative graded Grothendieck groups of the system of open immersions, we obtain a direct system $\big( K_0(\Var_{A_i}^{\dim}) \rightarrow K_0(\Var_{A_j}^{\dim}) \big)_{i \leq j \in I}$, each map of graded degree zero.

\begin{thm}\label{Spreading Out}
The natural graded ring homomorphism
\begin{align*}
    \phi: \varinjlim_{i \in I} K_0(\Var_{A_i}^{\dim}) \rightarrow K_0(\Var_A^{\dim})
\end{align*}
is an isomorphism. Moreover, it restricts to a map $\varinjlim_{i \in I} K_0(\Var_{A_i}^{\spr}) \rightarrow K_0(\Var_A^{\spr})$, which is an isomorphism.
\end{thm}

See \cite[Section $3.4$]{spreadingout} for an ungraded spreading out principle in the context of an arbitrary direct system of rings. In the graded context, one needs to be more careful keeping track of dimensions. 

\begin{proof}
If $X_i \rightarrow \mathop{\mathrm{Spec}}A_i$ is a finite type morphism with $\dim_{A_i} X_i = d$, then the pullback $X_i \times_{A_i} A \rightarrow A$ also satisfies $\dim_A (X_i \times_{A_i} A) = d$. It follows that the pullback $K_0(\Var_{A_i}^{\dim}) \rightarrow K_0(\Var_A^{\dim})$ and hence also the direct limit $\phi$ are of graded degree zero. Note that this is false if we had used the naive notion of relative dimension $\dim X_i - \dim A_i$ instead.

We now argue as in \cite{spreadingout}. Surjectivity of $\phi$ follows from the existence, for any $X \rightarrow A$, of an $A_i$-model $X_i \rightarrow A_i$ for some $i \in I$ by \cite[IV$.8.8.2$]{ega}.

For injectivity, we use the results that if $X_i$ and $Y_i$ are finite type schemes over $A_i$, then the natural map
\begin{align*}
    \varinjlim_{j \geq i} \Hom_{A_j}(Y_i \times_{A_i} A_j, X_i \times_{A_i} A_j) \rightarrow \Hom_A(Y_i \times_{A_i} A, X_i \times_{A_i} A)
\end{align*}
is a bijection. Moreover, given $f_i: Y_i \rightarrow X_i$, then the induced morphism $f:Y_i \times_{A_i} A \rightarrow X_i \times_{A_i} A$ is a closed (resp. open) immersion if and only if there exists some $j \geq i$ for which the induced morphism $f_j:Y_i \times_{A_i} A_j \rightarrow X_i \times_{A_i} A_j$ is  closed (resp. open immersion), by \cite[IV$.8.8.2$ and IV$.8.10.5$]{ega}.

For the final statement, we need simply note that the map $K_0(\Var_{A_i}^{\dim}) \rightarrow K_0(\Var_A^{\dim})$ sends $K_0(\Var_{A_i}^{\spr})$ to $K_0(\Var_{A}^{\spr})$.
\end{proof}

\subsection{Stacks}

Let $S$ be a base scheme. We introduce a graded Grothendieck group $\KstkS$ of algebraic stacks over $S$, as studied by Ekedahl \cite{ekedahlstacks} in the ungraded situation. The Grothendieck group $\KstkS$ will be a localization of $\KdimS$. 

All stacks in this section are assumed to be of finite type over $S$, with affine automorphism group schemes.

Let $\mathscr{X} \rightarrow S$ be a stack as above. Then $\mathscr{X}$ is locally covered by global quotient stacks $[X_i / GL_{n_i}]$, where $X_i$ is a scheme. We define the relative dimension $\dim_S \mathscr{X}$ to be $\sup_i \big( \dim_S X_i - \dim GL_{n_i})$ for any such covering.

\begin{defn}\label{K0Stk defn}
The graded Grothendieck group $\KstkS$ of stacks over $S$ is the graded abelian group with:
\begin{itemize}
    \item generators in degree $d$ the isomorphism classes $[\mathscr{X}]_d$ of stacks $\mathscr{X}$ of finite type over $S$, with $d \geq \dim_S \mathscr{X}$, and
    \item relations:
        \begin{enumerate}
            \item $\{ \mathscr{X} \}_d = \{ \mathscr{Y} \}_d + \{ \mathscr{X}-\mathscr{Y} \}_d$ is $\mathscr{Y}$ is a closed substack of $X$, and
            \item $\{ \mathscr{E} \}_d = \{ \mathscr{X} \times_S \mathbb{A}^n_S \}_d$ if $\mathscr{E} \rightarrow \mathscr{X}$ is a vector bundle of rank $n$.
        \end{enumerate}
\end{itemize}
\end{defn}

We use brackets in the notation to distinguish between classes in the Grothendieck group and stack quotients $[X/G]$. As usual, we will drop the subscript when the stack is placed in degree equal to its dimension. 

Note that relation $(2)$ is automatic only if we work only with schemes.

Observe, moreover, that $\KstkS$ is naturally a graded $\mathbb{Z}[\tau, \mathbb{L}]$-module, where $\tau$ acts by $\tau \cdot \{ \mathscr{X} \}_d = \{ \mathscr{X} \}_{d+1}$ and $\mathbb{L}$ acts by  $\mathbb{L} \cdot \{ \mathscr{X} \}_d = \{ \mathscr{X} \times_S \mathbb{A}^1_S \}_{d+1}$. If $S$ is of finite type over $k$, then $\KstkS$ is a module over $\Kstk$.

\begin{lem}\label{GLn torsor}
If $\mathscr{Y} \rightarrow \mathscr{X}$ is a $GL_n$-torsor, then $\{ \mathscr{Y} \rightarrow S \} = \{ GL_n \} \cdot \{ \mathscr{X} \rightarrow S \}$ in $\KstkS$. In particular, $\{ BGL_n \} = \{ GL_n \}^{-1}$ in $\Kstk$.
\end{lem}

\begin{proof}
There is a sequence of fibrations $\mathscr{Y} = \mathscr{Y}_n \rightarrow \mathscr{Y}_{n-1} \ldots \rightarrow \mathscr{Y}_1 \rightarrow \mathscr{Y}_0 = \mathscr{X}$, where $\mathscr{Y}_k$ is the stack of $k$ linearly independent vector bundles in the vector bundle associated to the $GL_n$-torsor. Then $\mathscr{Y}_k \rightarrow \mathscr{Y}_{k-1}$ is the complement of a vector subbundle of rank $k-1$ inside a vector bundle over $\mathscr{Y}_{k-1}$ of rank $n$. Using relation $(2)$ repeatedly, we obtain
\begin{align*}
    \{ \mathscr{Y} \} = (\mathbb{L}^n - \tau^n) \cdot (\mathbb{L}^n - \mathbb{L} \tau^{n-1}) \ldots \cdot (\mathbb{L}^n - \mathbb{L}^{n-1} \tau) \cdot \{ X \} = \{ GL_n \} \cdot \{ X \}.
\end{align*}
\end{proof}

\begin{prop}\label{localization}
There is a natural isomorphism
\begin{align*}
    \KdimS \otimes_{\mathbb{Z}[\tau, \mathbb{L}]} \mathbb{Z}[\tau, \mathbb{L}, [GL_n]_{n \geq 1}^{-1}] \rightarrow \KstkS,
\end{align*}
where $[GL_n] = (\mathbb{L}^n - \tau^n) \cdot \ldots (\mathbb{L}^n - \mathbb{L}^{n-1}\tau)$.
\end{prop}

\begin{proof}
Consider the natural map $\KdimS \rightarrow \KstkS$ of $\mathbb{Z}[\tau, \mathbb{L}]$-modules. We can extend it to 
\begin{align*}
    \KdimS \otimes_{\mathbb{Z}[\tau, \mathbb{L}]}, \mathbb{Z}[\tau, \mathbb{L}, [GL_n]_{n \geq 1}^{-1}] \rightarrow \KstkS
\end{align*}
by having $[GL_n]_{n \geq 1}^{-1}$ act as multiplication by $\{ B GL_n \}$ by Lemma \ref{GLn torsor}.

Any stack can be stratified by global quotient stacks $[X/GL_n]$ where $X$ is a scheme, so we need only define the inverse morphism on such classes. Send $\{ [X/GL_n] \}$ to $[X]/[GL_n]$. For precise details showing this is well-defined, see \cite[Theorem $1.2$]{ekedahlstacks}.
\end{proof}

We write the result of Proposition \ref{localization} as
\begin{align*}
    \KstkS \cong (T')^{-1} \KdimS,
\end{align*}
where $T'$ is the multiplicative set generated by $\mathbb{L}$, $(\mathbb{L}-\tau)^2$, and $[\mathbb{P}^n] = \tau^n + \ldots + \mathbb{L}^n$ for all $n \geq 1$. This is equivalent to the Proposition \ref{localization} by examining the factorization of $\{ GL_n \}$ in the proof of Lemma \ref{GLn torsor}.

Additionally inverting $\tau$ yields
\begin{align*}
    U^{-1} \KstkS \cong T^{-1} \KdimS,
\end{align*}
where $T$ is generated by $\tau \mathbb{L}$, $(\mathbb{L}-\tau)^2$, and $[\mathbb{P}^n] = \tau^n + \ldots + \mathbb{L}^n$ for all $n \geq 1$, and $U$ is generated by $\tau$.

The multiplicative set $T$ is symmetric in the classes $\tau$ and $\mathbb{L}$, so the involution $\mathbb{D}$ extends to its localization. In fact, $U^{-1} \KstkS \cong T^{-1} \KdimS$ is a quadratic extension of $T^{-1} \KsprS$:
\begin{align*}
    U^{-1} \KstkS \cong T^{-1} \KdimS \cong T^{-1} \KsprS \otimes_{\mathbb{Z}[\tau+\mathbb{L}, \tau \mathbb{L}]} \mathbb{Z}[\tau, \mathbb{L}].
\end{align*}

\subsection{Deligne-Mumford Stacks}\label{DM section}
By definition, $[X] \in T^{-1} \KsprS$ whenever $X$ is a regular scheme, proper over $S$. Our goal in the present subsection is to extend this containment to classes of regular Deligne-Mumford stacks $\mathscr{X}$, proper over $S$.

We accomplish this by realizing $\KstkS$ as a localization of a graded Grothendieck group $\KdmS$ of Deligne-Mumford stacks, exactly as for $\KdimS$.

\begin{defn}\label{K0DM defn}
The graded Grothendieck group $\KdmS$ of Deligne-Mumford stacks over $S$ is the graded abelian group with:
\begin{itemize}
    \item generators in degree $d$ the isomorphism classes $[\mathscr{X}]_d$ of Deligne-Mumford stacks $\mathscr{X}$ of finite type over $S$, with $d \geq \dim_S \mathscr{X}$, and
    \item relations:
        \begin{enumerate}
            \item $\{ \mathscr{X} \}_d = \{ \mathscr{Y} \}_d + \{ \mathscr{X}-\mathscr{Y} \}_d$ is $\mathscr{Y}$ is a closed substack of $X$, and
            \item $\{ \mathscr{E} \}_d = \{ \mathscr{X} \times_S \mathbb{A}^n_S \}_d$ if $\mathscr{E} \rightarrow \mathscr{X}$ is a vector bundle of rank $n$.
        \end{enumerate}
\end{itemize}
\end{defn}
An ungraded Grothendieck group $K_0'(\DM_k^{\dim})$ of Deligne-Mumford stacks is studied in \cite{berghK0} and \cite{bgll}, where the authors do not impose relation $(2)$.

The group $\KdmS$ is a module over $\mathbb{Z}[\tau, \mathbb{L}]$. We find a Bittner presentation for $\KdmS$, but only after modding out by $(\mathbb{L}-\tau)$-torsion elements.

\begin{thm}\label{K0DMbit}
Up to $(\mathbb{L}-\tau)$-torsion, $\KdmS$ admits a presentation with:
\begin{itemize}
    \item generators in degree $d$ the isomorphism classes $\{\mathscr{X}\}_d$ of regular, connected Deligne-Mumford stacks of finite type and proper over $S$, with $d \geq \dim_S \mathscr{X}$, and
    \item relations:
        \begin{enumerate}[label=(\arabic{enumi}')]
            \item $\{ \Bl_\mathscr{Y} \mathscr{X} \}_d - \{ E\}_d = \{\mathscr{X}\}_d - \{\mathscr{Y}\}_d$, where $\Bl_\mathscr{Y} \mathscr{X}$ is a stacky blow-up of $\mathscr{X}$ along a regular and connected $\mathscr{Y}$, with exceptional divisor $E$.
            \item $\{ \mathbb{P}(\mathscr{E}) \}_d = \{\mathscr{X} \times_S \mathbb{P}^{n-1}_S \}_d$ if $\mathscr{E} \rightarrow \mathscr{X}$ is a vector bundle of rank $n$.
        \end{enumerate}
\end{itemize}
\end{thm}

\begin{proof}
Denote by $K_0(\DM_S^{\bit})$ the graded abelian group with presentation as above. 

Let $K_0'(\DM_S^{\dim})$ and $K_0'(\DM_S^{\bit})$ be the analogous Grothendieck groups, defined without imposing relations $(2)$ and $(2')$. We claim that there is an isomorphism
\begin{align*}
   \Theta': K_0'(\DM_S^{\dim}) \rightarrow K_0'(\DM_S^{\bit}).
\end{align*}
This is essentially due to Bergh \cite[Theorem $1.1$]{berghK0}, except we need to keep track of dimension and in the relative setting. Alternatively, one can follow the argument of Theorem \ref{bittner presentation}. It works in this modified context, as versions of resolution of singularities, and the weak factorization theorem remain true for Deligne-Mumford stacks.

We now identify $K_0'(\DM_S^{\dim})$ with $K_0'(\DM_S^{\bit})$ and show conditions $(2)$ and $(2')$ are equivalent, up to $(\mathbb{L}-\tau)$-torsion. Let $\mathscr{E} \rightarrow \mathscr{X}$ be vector bundle of rank $n$. 

First we assume condition $(2')$. Then
\begin{align*}
    \{ \mathscr{E} \} =& \{ \mathbb{P}(\mathscr{E} \oplus 1) \} - \tau \cdot \{ \mathbb{P}(\mathscr{E}) \} \\
    =& ( \{ \mathbb{P}^n \} - \tau \cdot \{ \mathbb{P}^{n-1} \}) \cdot \{ \mathscr{X} \} \\
    =& \mathbb{L}^n \cdot \{ X \} \\
    =& \{ \mathscr{X} \times \mathbb{A}^n \},
\end{align*}
where we drop subscripts and gradings from the notation as convenient.

Now we assume condition $(2)$. There is a rank one vector bundle $\Bl_0 (\mathscr{E}) \rightarrow \mathbb{P}(\mathscr{E})$ obtained by blowing up the zero section of $\mathscr{E}$. We have
\begin{align*}
    \mathbb{L} \cdot \{ \mathbb{P}(\mathscr{E}) \} = \{ \Bl_0(\mathscr{E}) \} = \{ \mathscr{E} \}  - \tau^n \cdot \{ \mathscr{X} \} +  \tau \cdot \{ \mathbb{P}(\mathscr{E}) \}.
\end{align*}
Rearranging, we deduce
\begin{align*}
    (\mathbb{L}-\tau) \cdot \{ \mathbb{P}(\mathscr{E}) \} =& \{ \mathscr{E} \}  - \tau^n \cdot \{ \mathscr{X} \}\\ =& (\mathbb{L}^n - \tau^n) \cdot \{ \mathscr{X} \} = (\mathbb{L}-\tau) \cdot \{ \mathscr{X} \times \mathbb{P}^{n-1} \}.
\end{align*}
\end{proof}

\begin{cor}
Up to $\mathbb{Z}[\tau, \mathbb{L}]$-torsion, $\KdmS$ admits a presentation with:
\begin{itemize}
    \item generators in degree $d$ the isomorphism classes $[\mathscr{X}]$ of smooth, connected Deligne-Mumford stacks, which are proper, of finite type, and of relative dimension $d$ over $S$, and
    \item relations
        \begin{align*}
            [\Bl_\mathscr{Y} \mathscr{X}] = [\mathscr{X}] +\tau \mathbb{L}\cdot (\tau^{k-1} + \ldots +\mathbb{L}^{k-1}) \cdot [\mathscr{Y}],
        \end{align*}
            where $\Bl_\mathscr{Y} \mathscr{X}$ is a stacky blow-up of $\mathscr{X}$ along a regular and connected $\mathscr{Y}$ of codimension $k+1$ with exceptional divisor $E$.
        \begin{align*}
            [\mathbb{P}(\mathscr{E})] = (\tau^{k}+ \tau^{k-1} \mathbb{L} +  \ldots + \mathbb{L}^{k}) \cdot [\mathscr{Y}]
        \end{align*} 
        for $\mathscr{E} \rightarrow \mathscr{Y}$ a vector bundle of rank $k+1$.
\end{itemize}
\end{cor}

\begin{proof}
We follow the same argument relating $K_0(\DM_S^{\bit})$ to the above presentation as in Corollary \ref{bit module}. The only difference is that stacky blow-ups can consist of a composition not only of a smooth blow-up, but also of a root construction in the exceptional divisor. We need to show that root constructions do not alter classes in $K_0(\DM_S^{\dim})$, up to $(\mathbb{L}-\tau)$-torsion.

Consider an $r$th root construction $\pi: \tilde{\mathscr{X}} \rightarrow \mathscr{X}$ of a divisor $D$ in $\mathscr{X}$, with preimage $\tilde{D} = \pi^{-1}(D)$. The morphism $\pi$ is an isomorphism away from $D$.  We analyze what happens near $D$ by a local computation. Suppose $D = (f=0)$; then
\begin{align*}
    \tilde{\mathscr{X}} = [\mathop{\mathrm{Spec}} \big( \mathscr{O}_\mathscr{X}[t]/(t^r - f) \big) / \mu_r ],
\end{align*}
where a generator of $\mu_r$ acts as multiplication by a $r$th root of unity on $t$.

Over $D = (f=0)$, this looks like
\begin{align*}
    [\mathop{\mathrm{Spec}} \big( \mathscr{O}_\mathscr{X}[t]/t^r \big) / \mu_r ] = D \times [\mathop{\mathrm{Spec}} \big( k[t]/t^r \big) / \mu_r],
\end{align*}
which has reduction $D \times B \mu_r$, where $B \mu_r$ is the classifying stack of the cyclic group $\mu_r$ of order $r$. As reduction does not change classes in the Grothendieck group, it remains to show that $\{ B \mu_r \} = 1$ in $\Kdm$, up to $(\mathbb{L}-\tau)$-torsion. This is precisely the content of Lemma \ref{B unity}.
\end{proof}

\begin{lem}\label{B unity}
  $\{ B \mu_r \} = 1$ in $\Kdm$, up to $(\mathbb{L}-\tau)$-torsion. 
\end{lem}

\begin{proof}
We compute
\begin{align*}
    \mathbb{L} \cdot \{ B \mu_r \} = \{ [\mathbb{A}^1 / \mu_r] \} = (\mathbb{L}-\tau) + \tau \cdot \{ B \mu_r \}.
\end{align*}
We obtain the result by rearranging.
\end{proof}

Let $K_0(\DM_S^{\spr})$ be the subgroup of $\KdmS$ generated by the classes $[\mathscr{X}]$ of smooth Deligne-Mumford stacks $\mathscr{X}$, proper over $S$. It is a module over $\mathbb{Z}[\tau+\mathbb{L}, \tau \mathbb{L}]$. We deduce the following structural result, exactly as in Theorem \ref{str iso}.

\begin{thm}\label{str iso dm}
There exists a natural isomorphism
\begin{align*}
    K_0(\DM_S^{\spr}) \otimes_{\mathbb{Z}[\tau+ \mathbb{L}]} \mathbb{Z}[\tau, \mathbb{L}] \rightarrow  K_0(\DM_S^{\dim}),
\end{align*}
up to $(\mathbb{L}-\tau)$-torsion.
\end{thm}

There exist natural morphisms
\begin{align*}
    \KdimS \rightarrow \KdmS \rightarrow \KstkS.
\end{align*}
Inverting by the multiplicative set $T'$ generated by $\mathbb{L}$, $(\mathbb{L}-\tau)^2$, and $[\mathbb{P}^n]$ for $n \geq 1$, we obtain morphisms
\begin{align*}
    (T')^{-1}\KdimS \rightarrow (T')^{-1}\KdmS \rightarrow \KstkS.
\end{align*}
Using Proposition \ref{localization}, it is straightforward to check these are all isomorphisms.

Inverting further by $\tau$, we obtain isomorphisms
\begin{align*}
    T^{-1}\KdimS \rightarrow T^{-1}\KdmS \rightarrow U^{-1}\KstkS.
\end{align*}
Thus, $T^{-1}\KdimS \cong T^{-1}\KdmS$ is a quadratic extensions of both $T^{-1}\KsprS$ and $T^{-1} K_0(\DM_S^{\spr})$. As the former is clearly contained in the latter, we obtain an isomorphism
\begin{align*}
    T^{-1}\KsprS \cong T^{-1} K_0(\DM_S^{\spr}).
\end{align*}

\begin{cor}\label{DM in spr}
If $\mathscr{X}$ is a regular Deligne-Mumford stack, proper and of finite type over $S$, then the class $\{ \mathscr{X} \} \in \KstkS$ lies in the submodule $T^{-1} \KsprS$. 
\end{cor}

\section{Gluing Map}\label{gluing map}

We construct a gluing map between relative Grothendieck groups to measure the interaction between their respective involutions. We explore concrete geometric applications in \ref{subsection applications}. 

\subsection{Construction}

Throughout, let $X$ be a scheme and let $i:Y \rightarrow X$ be a closed embedding with complementary open embedding $j: U \rightarrow X$.

It is straightforward to verify that there is a split short exact sequence of graded abelian groups
\begin{center}
\begin{tikzcd}
    K_0(\Var_U^{\dim}) \arrow{r}{j_!} & K_0(\Var_X^{\dim}) \arrow{r}{i^*} \arrow[bend left=33]{l}{j^*}  & K_0(\Var_Y^{\dim}) \arrow[bend left=33]{l}{i_!}.
\end{tikzcd}
\end{center}
Here, $j_!$ and $j^*$ are of graded degree zero and $i^*$ and $i_!$ are of graded degree $c$ and $-c$, respectively, where $c = \codim_X Y$.

The behavior of the smooth and proper submodules with respect to this diagram is more interesting.

Observe that $i_!(K_0(\Var_Y^{\spr})) \subset K_0(\Var_X^{\spr})$ and $j^* \big( K_0(\Var_X^{\spr} \big)) \subset K_0(\Var_U^{\spr})$. The analogous results fail for $j_!$ and $i^*$.

\begin{defn}
The gluing morphism $\mathfrak{g}_{U}^Y: K_0(\Var_U^{\spr}) \rightarrow K_0(\Var_Y^{\spr})$ of $Y$ to $U$ via $X$ is the composition
\begin{align*}
    K_0(\Var_U^{\spr}) \xrightarrow[]{j_!} K_0(\Var_X^{\dim}) \xrightarrow[]{-\pi_2} K_0(\Var_X^{\dim}) \xrightarrow[]{i^*} K_0(\Var_Y^{\dim}).
\end{align*}
Since the image of $-\pi_2 \circ j_!$ lies in $K_0(\Var_X^{\spr})$ and is supported on $Y$, we have $\mathfrak{g}_U^Y(K_0(\Var_U^{\spr})) \subset K_0(\Var_Y^{\spr})$. Observe that $\mathfrak{g}_{U}^Y$ is of graded degree $c-1$. It is $\mathbb{Z}[\tau+ \mathbb{L}, \tau \mathbb{L}]$-linear and, moreover, $\Kspr$-linear if $X$ is of finite type over $k$.
\end{defn}

By definition, we have
\begin{align*}
    K_0(\Var_X^{\spr}) = \big\{ \hspace{.1cm} j_!(\alpha) + i_!(\mathfrak{g}_U^Y(\alpha) + \beta) \hspace{.15cm} | \hspace{.15cm} \alpha \in K_0(\Var_U^{\spr}), \hspace{.1cm} \beta \in K_0(\Var_Y^{\spr}) \hspace{.1cm} \big\}.
\end{align*}
The involution $\mathbb{D}_X$ on $K_0(\Var_X^{\dim})$ is therefore also determined by the involutions $\mathbb{D}_U$ and $\mathbb{D}_Y$, together with the gluing morphism $\mathfrak{g}_{U}^Y$.

\begin{defn}
The (absolute) gluing morphism $\mathfrak{g}_U: K_0(\Var_U^{\spr}) \rightarrow K_0(\Var_k^{\spr})$ of a scheme $U$ of finite type over $k$ is the composition
\begin{align*}
    K_0(\Var_U^{\spr}) \xrightarrow[]{(a_U)_!} \Kdim \xrightarrow[]{-\pi_2} \Kspr,
\end{align*}
where $a_U:U \rightarrow \mathop{\mathrm{Spec}}k$.
\end{defn}

The relative gluing morphism $\mathfrak{g}_U^Y$ is a refinement of the absolute gluing morphism $\mathfrak{g}_U$, in the following sense.

\begin{prop}\label{gluing refinement}
If $j:U \rightarrow X$ is an open embedding with $Y = X-U$ and $X$ proper, then
\begin{align*}
    (a_Y)_! \circ \mathfrak{g}_U^Y = \mathfrak{g}_U,
\end{align*}
where $a_X: X \rightarrow \mathop{\mathrm{Spec}}k$.
\end{prop}

\begin{proof}
We need to analyze the diagram
\[
\begin{tikzcd}
  K_0(\Var_U^{\spr}) \arrow[r, "j_!"] \arrow[dr,bend right, "(a_U)_!"] & K_0(\Var_X^{\dim}) \arrow[r, "-\pi_2"] \arrow[d, "(a_X)_!"] & K_0(\Var_X^{\spr}) \arrow[r, "i^*"] \arrow[d, "(a_X)_!"] & K_0(\Var_Y^{\spr}) \arrow[dl, bend left, "(a_Y)_!"] \\
  & \Kdim \arrow[r, "-\pi_2"] & \Kspr
\end{tikzcd}
\]
The first triangle is clearly commutative and the rectangle is commutative by Lemma \ref{proper}, as $X$ is proper. The second rectangle is commutative when restricted to $\im( -\pi_2 \circ j_!) \subset \im(i_!)$ because $i^* i_! = \id$. The result follows.
\end{proof}

\subsection{Applications}\label{subsection applications}

In this subsection, we parse the geometric data captured by gluing morphisms $\mathfrak{g}_U^Y$. Our results can be obtained more directly by careful arguments relying on the weak factorization theorem. It is our hope, however, that the present systematic approach is more illuminating.

We start with the identity $\mathbbm{1}_U$ of a smooth variety $U$.

\begin{prop}\label{genl gluing application}
Let $U$ be a smooth variety and let $U \rightarrow X$ be an open embedding, with $D = X - U$. If $f: \Tilde{X} \rightarrow X$ is proper and an isomorphism over $U$, with $\Tilde{X}$ smooth and $f^{-1}(D) = \cup_{i \in I} \Tilde{D}_i$ an snc divisor, then
\begin{align*}
    \mathfrak{g}_U^D(\mathbbm{1}_U) = \sum_{\emptyset \neq J \subset I} (-1)^{|J|-1} [\Tilde{D}_J \times \mathbb{P}^{|J|-1}] \in K_0(\Var_D^{\spr}),
\end{align*}
where $\tilde{D}_J = \cap_{j \in J} \tilde{D}_j$ for $J \subset I$.
\end{prop}

\begin{proof}
This follows immediately from the following decomposition of $[U \rightarrow X]$ into its smooth and proper components inside $K_0(\Var_X^{\dim})$, using Example \ref{P^k example}:
\begin{align*}
    [U \rightarrow X] =& \sum_{J \subset I} (-1)^{|J|} \tau^{|J|} [\Tilde{D}_J] \\
    =& \bigg( \sum_{J \subset I} (-1)^{|J|-1} \tau \mathbb{L} \cdot [\Tilde{D}_J \times \mathbb{P}^{|J|-2}] \bigg) \\ &- \tau \cdot \bigg( \sum_{\emptyset \neq J \subset I} (-1)^{|J|-1} [\Tilde{D}_J \times \mathbb{P}^{|J|-1}] \bigg).
\end{align*}
\end{proof}

\begin{cor}[cf. Theorem \ref{intro glue thm}]\label{paper glue thm}
Let $U$ be a smooth variety. Choose a smooth compactification $X$ of $U$ whose boundary $D = X - U$ is an snc divisor $D = \cup_{i \in I} D_i$. Then
\begin{align*}
    \mathfrak{g}_U(\mathbbm{1}_U) = \sum_{\emptyset \neq J \subset I} (-1)^{|J|-1} [D_J \times \mathbb{P}^{|J|-1}] \in \Kspr.
\end{align*}
In particular, this class depends only on $U$.
\end{cor}

\begin{cor}\label{sm boundary}
Suppose a smooth variety $U$ admits two smooth compactifications with smooth boundary divisors $D$ and $E$, respectively. Then $[D] = [E] \in \Kspr$. Hence, $D$ and $E$ are birational and have the same Hodge numbers.
\end{cor}

Let's pause to give a few examples of distinct compactifications with smooth boundary divisors.

\begin{exmp}
A standard flip cuts a smooth divisor isomorphic to a $\mathbb{P}^k$-bundle over $\mathbb{P}^\ell$ out of a smooth proper variety $X$ and replaces it with a new smooth divisor isomorphic to a $\mathbb{P}^\ell$-bundle over $\mathbb{P}^k$. As predicted by the Corollary \ref{sm boundary}, these divisors have the same class $[\mathbb{P}^k \times \mathbb{P}^\ell]$ in $K_0(\Var_k^{\dim})$.
\end{exmp}

\begin{exmp}
Let $U$ be the smooth toric threefold whose maximal dimensional cones are the $\mathbb{R}_{>0}$-linear spans of
\begin{align*}
    v_1 v_2 v_5, \hspace{.2cm} v_2 v_3 v_5, \hspace{.2cm} v_3 v_4 v_5, \hspace{.2cm} v_1 v_4 v_5, 
\end{align*}
where $v_1=e_1$, $v_2=e_2$, $v_3=-e_1$, $v_4=-e_2+e_3$, and $v_5=-e_3$ inside the standard lattice $\mathbb{Z}^3$.

For each integer $a \in \mathbb{Z}$, there is a smooth toric compactification $X_a$ of $U$ which adds new maximal dimensional cones spanned by
\begin{align*}
    v_1 v_2 v_6, \hspace{.2cm} v_2 v_3 v_6, \hspace{.2cm} v_3 v_4 v_6, \hspace{.2cm} v_1 v_4 v_6, 
\end{align*}
where $v_6 = -ae_1+e_3$.

One computes that the boundary divisor $D_a$ is the Hirzebruch surface $\mathbb{P}(\mathscr{O}_{\mathbb{P}^1} \oplus \mathscr{O}_{\mathbb{P}^1}(a))$. These Hirzebruch surfaces all represent to same class $(\mathbb{L}+\tau)^2$ in $K_0(\Var_k^{\dim})$ but are pairwise non-isomorphic.
\end{exmp}

If we allow $U$ to have singularities, the situation becomes more complicated. Nevertheless, we have the following result.

\begin{prop}
Let $U$ be a variety and let $V \rightarrow U$ be a resolution of $U$. If $X$ is a smooth compactification of $V$ with snc boundary $D = \cup_{i \in I} D_i$, then
\begin{align*}
  \mathfrak{g}_U(\pi_1[U]) = \sum_{\emptyset \neq J \subset I} (-1)^{|J|-1} [D_J \times \mathbb{P}^{|J|-1}] \in \Kspr/(\tau \mathbb{L}) = \mathbb{Z}[\Bir_k].
\end{align*}
In particular, this value depends only on $U$.
\end{prop}

\begin{proof}
We have $\pi_1[U] = [V]$ mod $\tau \mathbb{L}$ in $K_0(\Var_U^{\spr})$. The statement then follows from Corollary \ref{paper glue thm}, as the gluing morphism $\mathfrak{g}_U$ is $\Kspr$-linear.
\end{proof}

\begin{exmp}
Consider the gluing map
\begin{align*}
    \mathfrak{g} = \mathfrak{g}_{k((t))}^k: K_0(\Var_{k((t))}^{\spr}) \rightarrow \Kspr
\end{align*}
when $X = \mathop{\mathrm{Spec}} k[[t]]$, with $U = \mathop{\mathrm{Spec}}k((t))$ and $Y = \mathop{\mathrm{Spec}}k$. 

We can describe the class $\mathfrak{g}([\mathcal{X}])$ explicitly when $\mathcal{X}$ is a smooth and proper variety over the field $k((t))$. Indeed, pick a regular model $f: \mathscr{X} \rightarrow \mathop{\mathrm{Spec}}k[[t]]$ such that $f$ is proper, the general fiber is isomorphic to $\mathcal{X}$, and the special fiber is an snc divisor $\cup_{i \in I} d_i D_i$ over $k$. Then we have
\begin{align*}
    \mathfrak{g}([\mathcal{X}]) = \sum_{\emptyset \neq J \subset I} (-1)^{|J|-1} [D_J \times \mathbb{P}^{|J|-1}].
\end{align*}
From this description, it is immediate that the $\mathbb{Z}[\tau, \mathbb{L}]$-linear extension of $\mathfrak{g}$ to a map $K_0(\Var_{k((t))}^{\dim}) \rightarrow \Kdim$ coincides with the volume morphism

\begin{align*}
    \Vol: K_0(\Var_{k((t))}^{\dim}) \rightarrow \Kdim
\end{align*}
discussed at various levels of generality in \cite{stablebirational}, \cite{kontsevichtschinkel}, \cite{K^dim}. Its reduction
\begin{align*}
    \Vol: \mathbb{Z}[\Bir_{k((t))}] \rightarrow \mathbb{Z}[\Bir_k]
\end{align*}
is the key ingredient in proving that rationality is stable under specialization.
\end{exmp}

\section{\texorpdfstring{$\lambda$}{λ}-Structures on $K_0$-Rings}
We will require the language of $\lambda$-rings in Section \ref{kapranov zeta section} and especially the equality of two natural $\lambda$-structures on $\Kstk$ in Subsection \ref{sym section}.

\subsection{Basic Theory}
We first provide a rapid introduction to $\lambda$-rings. For more information, see \cite{atiyahtall} or \cite{larsenluntsrationality}.
\begin{defn}
A $\lambda$-structure on a ring $A$ is a sequence of maps $\lambda^m: A \rightarrow A$ for $m \geq 0$ such that, for all $x,y \in A$ and $m \geq 0$,
\begin{align*}
    \lambda^0(x) = 1, \hspace{.5cm} \lambda^1(x) = x, \hspace{.5cm} \lambda^m(x+y) = \sum_{i+j=m}\lambda^i(x) \cdot \lambda^j(y).
\end{align*}

A $\lambda$-ring is a ring equipped with a $\lambda$-structure. A $\lambda$-ring $A$ is graded if $A$ is a graded ring and $\lambda^m(A_d) \subset A_{md}$ for all $m$ and $d$. A (graded) $\lambda$-homomorphism between (graded) $\lambda$-rings is a (graded) ring homomorphism which commutes with the $\lambda$-operations on the source and target.
\end{defn}

\begin{exmp}
The integers $\mathbb{Z}$ admit a $\lambda$-structure $\lambda^i(n) = \binom{n+i-1}{i}$.
\end{exmp}

\begin{exmp}
Consider, more generally, the Grothendieck group $K(X)$ of vector bundles on a variety $X$. This is the abelian group generated by isomorphism classes of vector bundles on $X$, modulo the relations $[E] = [F]+[G]$ for every short exact sequence $0 \rightarrow F \rightarrow E \rightarrow G \rightarrow 0$. It admits a $\lambda$-structure $\lambda^i([E]) = [\Sym^i E]$.
\end{exmp}

We will use a variant of this construction.

\begin{exmp}\label{K bar}
For a variety $X$, consider the abelian group $\overline{K}(X)$ generated by isomorphism classes of vector bundles on $X$, modulo the relations $[E] = [F] + [G]$ when $E$ and $F \oplus G$ are isomorphic vector bundles. It admits a $\lambda$-structure $\lambda^i([E]) = [\Sym^i E]$. Unlike $K(X)$, it also admits a global sections map
\begin{align*}
    h^0: \overline{K}(X) \rightarrow \mathbb{Z}, \hspace{.5cm} [E] \mapsto \dim H^0(X, E)
\end{align*}
\end{exmp}

\begin{prop}\label{K0 lambda defn}
The graded Grothendieck ring of varieties $\Kdim$ admits a graded $\lambda$-ring structure defined by
\begin{align*}
    \lambda^m([X]_d) = [\Sym^m X]_{md}
\end{align*}
for quasi-projective varieties $X$ and integers $d \geq \dim X$. Moreover, 
\begin{align*}
    \lambda^m(\tau \alpha) = \tau^m \lambda^m(\alpha), \hspace{.5cm} \lambda^m(\mathbb{L} \alpha) = \mathbb{L}^m \lambda^m(\alpha)
\end{align*}
for all $\alpha \in \Kdim$.
\end{prop}

\begin{proof}
Let $Y \subset X$ be a closed subscheme with complement $U = X - Y$. Stratifying the symmetric power variety $\Sym^m X$ according to how many of its points lie in $U$ and $Y$ respectively, we see
\begin{align*}
    [\Sym^m X]_{md} = \sum_{i+j=m} [\Sym^i U]_{id} \cdot [\Sym^j Y]_{jd}. 
\end{align*}
This identity shows both that $\lambda^m$ is well-defined and that it is a $\lambda$-structure on $\Kdim$. 

The identity $\lambda^m(\tau \alpha) = \tau^m \lambda(\alpha)$ is clear, and $\lambda^m(\mathbb{L} \alpha) = \mathbb{L}^m \lambda(\alpha)$ follows from an argument of Totaro, as in \cite[Lemma $4.4$]{gottschehilbert}.
\end{proof}

\begin{rmk}
If $X$ isn't quasi-projective, then the quotient $\Sym^m X = X^m / S_m$ need not be defined as a scheme. It is, however, an algebraic space. As any algebraic space is generically a scheme, the naive Grothendieck ring $K_0(\Spc_k^{\dim})$ of algebraic spaces can be naturally identified with $\Kdim$. In this sense, we may always view $\lambda^m([X]_d)$ as equal to $[\Sym^m X]_{md}$.
\end{rmk}

Let $A$ be a $\lambda$-ring. An ideal $I$ of $A$ is called a $\lambda$-ideal if it is preserved by the $\lambda$-operations, in which case $A/I$ is naturally a $\lambda$-ring, and the quotient map $A \rightarrow A/I$ a $\lambda$-homomorphism. 

For example $(\tau)$ is a $\lambda$-ideal of $\Kdim$. The $\lambda$-structure on the quotient $\Kdim / (\tau) \cong \mathbb{Z}[\Bir_k]$ is given by taking symmetric powers of birational types.

There exist Adams operations $\Psi^n: A \rightarrow A$ for each $n \geq 1$. These are natural operations which are defined by universal polynomials in the $\lambda$-operations $\lambda^i$.

We can package together some of the axioms of a $\lambda$-structure by saying that $A$ is equipped with a group homomorphism
\begin{align*}
    \lambda_t: A \rightarrow (1+t A[[t]])^\times, \hspace{.5cm} a \mapsto \sum_{m \geq 0} \lambda^m(a) \cdot t^m.
\end{align*}
In fact, $(1+tA[[t]])^\times$ admits a natural $\lambda$-ring structure. Here, addition is given by multiplication of power series. Multiplication $\circ$ and the $\lambda$-operations are uniquely characterized by the identities
\begin{align*}
    \bigg( \prod_{i=1}^n (1+a_i t) \bigg) \circ \bigg( \prod_{j=1}^m (1+b_j t) \bigg) = \prod_{i=1}^n \prod_{j=1}^m (1+a_i b_j t ),
\end{align*}
and
\begin{align*}
    \lambda^m \bigg( \prod_{i=1}^n (1+a_i t) \bigg) = \prod_{J \subset \{1 \ldots n\}, \| J\|=m} (1+t \prod_{j \in J}a_j).
\end{align*}
Nevertheless, $\lambda_t$ is not a $\lambda$-homomorphism unless $A$ is a so-called special $\lambda$-ring. For example, $\Kdim$ is not special.

\begin{defn}
The graded Kapranov zeta function of a variety $X$ is 
\begin{align*}
    Z_X^{\dim}(t) = \lambda_t([X]) = \sum_{m \geq 0}[\Sym^m X] \cdot t^m \in \Kdim[[t]].
\end{align*}
The (ungraded) Kapranov zeta function $Z_X(t) \in K_0(\Var_k)[[t]]$ of $X$, introduced in \cite{kapranov}, is its reduction modulo $\tau-1$.
\end{defn}

\subsection{$\lambda$-Structures on $\Kstk$}
Consider the graded Grothendieck ring of stacks
\begin{align*}
    \Kstk \cong \Kdim[\mathbb{L}^{-1}, (\mathbb{L}^n-\tau^n)^{-1}_{n \geq 1}].
\end{align*} 
It inherits a graded $\lambda$-structure by applying the following Lemma, where we take $A = \Kdim$, $r_1 = \mathbb{L}$, $r_2 = \tau$, and $S$ the multiplicative set generated by $x_1$ and $x_1^m - x_2^m$ for all $m \geq 1$.

\begin{lem}\label{lambda ext}
Let $A$ be a $\lambda$-ring such that $\lambda_t(1) = (1-t)^{-1}$ and let $r_1, \ldots r_n \in A$ be such that $\lambda_t(r_ix) = \lambda_{r_it}(x)$ for $i = 1, \ldots , n$ and all $x \in A$. If $S$ is a set of integer polynomials in $n$ variables closed under multiplication and substitution $(x_1, \ldots x_n) \mapsto (x_1^m, \ldots x_n^m)$ for all $m \geq 1$ and $S' = \{ f(r_1, \ldots , r_n) | f \in S \}$, then there is a unique $\lambda$-structure on the localization $A' = (S')^{-1}A$ for which $A \rightarrow A'$ is a $\lambda$-ring homomorphism and for which $\lambda_t(r_i x) = \lambda_{r_i t}(x)$ for $i = 1, \ldots , n$ and all $x \in A'$.

Moreover, if $A$ is a graded $\lambda$-ring and each $r_i$ and each $f(r_1, \ldots , r_n)$ is homogeneous, then $A'$ is a graded $\lambda$-ring.
\end{lem}

For ease of language, we will hereafter say that $r_1, \ldots , r_n$ are $\lambda$-linear in $A$ if $\lambda_t(r_i x) = \lambda_{r_i t}(x)$ for $i = 1, \ldots , n$ and all $x \in A$.

\begin{proof}
Special care is required to keep track of notation in the ring $(1+t A[[t]])^\times$, which we will make use of throughout the proof. The symbol $\circ$ denotes multiplication in the ring, the usual multiplicative symbols $\cdot$, $\bullet^{-1}$, and $\prod$ refer to addition in the ring (multiplication of power series), and the usual addition symbols $+$ and $\Sigma$ refer to addition of power series.

We compute, for $x \in A$ and $i =1, \ldots , n$,
\begin{align*}
    \lambda_t(r_i x) = \lambda_{r_i t}(x) =  (1+r_i t) \circ \lambda_t(x) = (\lambda_{-t}(r_i)^{-1}) \circ \lambda_t(x) = \sigma_t(r_i) \circ \lambda_t(x),
\end{align*}
where we denote $\sigma_t(r_i) = \lambda_{-t}(r_i)^{-1}$. In fact, if $R$ is the subring of $A$ generated by $r_1, \ldots , r_n$, we can immediately show that the map
\begin{align*}
    \sigma_t: R \rightarrow (1+t A[[t]])^\times, \hspace{.5cm} r \mapsto \lambda_{-t}(r)^{-1}
\end{align*}
is a ring homomorphism which makes
\begin{align*}
    \lambda_t: A \rightarrow (1+t A[[t]])^\times
\end{align*}
into a $R$-module homomorphism. To prove the lemma, we need to show that $\lambda_t$ extends uniquely to a $(S')^{-1}R$-module homomorphism
\begin{align*}
    \lambda_t: A' \rightarrow (1+tA'[[t]])^\times.
\end{align*}
For this, it suffices to show that the image of $\sigma_t(f(r_1, \ldots r_n))$ in $(1+tA'[[t]])^\times$ is invertible for all $f \in S$..

Write $f(x) = \sum_I a_I x_1^{i_1} \ldots x_n^{i_n}$. We construct an inverse $\phi(t) = 1 + \sum_{n > 0} b_n t^n$ to $\sigma_t(f(r_1, \ldots r_n))$ by determining the coefficients $b_n$ inductively. For this, let's compute
\begin{align*}
     \sigma_t(f(r_1, \ldots r_n)) \circ \phi(t) 
        =& \prod_I \sigma_t( r_1^{i_1} \ldots r_n^{i_n} )^{a_I} \circ \phi(t) \\
        =& \prod_I \phi(r_1^{i_1} \ldots r_n^{i_n} t)^{a_I} \\
        =& \prod_I (1+\sum_{n>0} b_n r_1^{n i_1} \ldots r_n^{n i_n} t^n)^{a_I}.
\end{align*}
For the second equality, we used repeatedly the identity
\begin{align*}
    \sigma_t(r_i) \circ \phi(t) = (1+r_i t) \circ \phi(t) = \phi(r_i t).
\end{align*}

The coefficient for $t^n$ in this expression is a sum, consisting of $\sum_I a_I b_n r_1^{n i_1} \ldots r_n^{n i_n} = b_n f(r_1^n, \ldots r_n^n)$ and other terms, each of which contains $b_i$ for some $i<n$. Assuming $b_i$ for $i<n$ have already been chosen uniquely, we may solve uniquely for $b_n$ inside of $A'$ so as to force this coefficient to vanish.

The final assertion is clear from our proof if we keep track of graded degree.
\end{proof}

\begin{rmk}\label{lambda extn rmk}
Let us explicitly note for future reference that, in the situation of Lemma \ref{lambda ext}, our proof shows 
\begin{align*}
    \lambda_t: A' \rightarrow (1+tA'[[t]])^\times
\end{align*}
is an $(S')^{-1}R$-module homomorphism, where $R$ is the ring generated by $r_1,\ldots r_n$.
\end{rmk}

On the other hand, there is also a more natural stack-theoretic notion of symmetric product $\Symm$. For example, if $X$ is a scheme, then $\Symm^m(X)$ is the stack quotient $[X^m/S_m]$. See \cite{ekedahlgroup} for further details and a proof of the following, which follows along the lines of Proposition \ref{K0 lambda defn}.

\begin{prop}
The graded Grothendieck ring of stacks $\Kstk$ admits a graded $\lambda$-ring structure defined by $\sigma^m \{ \mathscr{X} \}_d = \{ \Symm^m \mathscr{X} \}_{md}$.
\end{prop}

To show that the two lambda-structures coincide, we will need the following theorem of Ekedahl \cite[Theorem $4.3$]{ekedahlgroup}.

\begin{thm}\label{B S_m}
$\{ BS_m \} = 1$ in $\Kstk$ for all $m \geq 1$.
\end{thm}
Here $B S_m$ is the classifying stack of the symmetric group $S_m$. We already saw an analogous result holds for cyclic groups and hence for all finite abelian groups in Lemma \ref{B unity}.

\begin{prop}\label{coincide}
The two graded $\lambda$-structures $\lambda^m$ and $\sigma^m$ on $\Kstk$ coincide.
\end{prop}

\begin{proof}
We need to show that the two group homomorphisms
\begin{align*}
    \lambda_t, \sigma_t: \Kstk \rightarrow (1+\Kstk[[t]])^\times
\end{align*}
are equal. Note that $\tau$ and $\mathbb{L}$ are both $\lambda$-linear in $\Kstk$ by Proposition \ref{K0 lambda defn} and Lemma \ref{lambda ext}. It is clear that $\tau$ is $\sigma$-linear in $\Kstk$; we argue $\mathbb{L}$ is as well. Indeed, we reduce immediately to checking $\{ \Symm^m (\mathscr{X} \times \mathbb{A}^1) \} = \mathbb\{ \Symm^m(\mathscr{X}) \times \mathbb{A}^m \}$. But this is clear, as $\Symm^m(\mathscr{X} \times \mathbb{A}^1) \rightarrow \Symm^m(\mathscr{X})$ is a vector bundle of rank $m$.

By Remark \ref{lambda extn rmk}, both $\lambda_t$ and $\sigma_t$ are $(T')^{-1} \mathbb{Z}[\tau, \mathbb{L}]$-linear. We therefore need only check that they agree on the class $\{ X \}$ of a variety $X$.

For a partition $\alpha$ of $m$, we denote by $\Sym^\alpha X$ the locally closed subvariety of $\Sym^m X$ with tuples consisting of $\lambda_1$ points of multiplicity $1$, $\lambda_2$ points of multiplicity two, and so on. Then
\begin{align*}
    \lambda^m( \{ X \} ) = \{ \Sym^m X \}_= \sum_{\alpha \vdash m} \{ \Sym^\alpha X \}.
\end{align*}
On the other hand, let $\Symm^\alpha X$ be the preimage of $\Sym^\alpha X$ under the map $\Symm^m X \rightarrow \Sym^m X$. Then we have 
\begin{align*}
    \Symm^\alpha X \cong \Sym^m X \times B(S_1^{\times(\alpha_1)}) \times B(S_2^{\times(\alpha_2)}) \ldots \times B(S_m^{\times(\alpha_m)}).
\end{align*}
Combining this with Theorem \ref{B S_m}, we conclude
\begin{align*}
    \sigma^m( \{ X \} ) = \{ \Symm^m X \} = \sum_{\alpha \vdash m} \{ \Symm^\alpha X \} = \sum_{\alpha \vdash m} \{ \Sym^\alpha X \}
    = \lambda^m( \{ X \}).
\end{align*}
\end{proof}

\begin{rmk}\label{sym = symm}
In fact, this same argument shows
\begin{align*}
    \{ \Sym^m X \} = \{ \Symm^m X \} \in K_0(\Stk_{\Sym^m X}^{\dim}).
\end{align*}
\end{rmk}

Applying Lemma \ref{lambda ext}, the $\lambda$-structure on $\Kstk$ extends to a graded $\lambda$-structure on
\begin{align*}
    U^{-1} \KstkS \cong T^{-1} \KdimS.
\end{align*}

\begin{prop}[cf. Theorem \ref{Sym commute D torsion}]
We have the identity
\begin{align*}
    \lambda^m \circ \mathbb{D} = \mathbb{D} \circ \lambda^m
\end{align*}
in $U^{-1}K_0(\Stk_k^{\dim})$ for all $m \geq 1$. Therefore, the same holds in $\Kdim$, up to $T$-torsion classes.
\end{prop}

\begin{proof}
Consider the two graded $\lambda$-structures on $U^{-1}K_0(\Stk_k^{\dim})$ given by $\lambda^m$ and by $(\lambda')^m = \mathbb{D} \circ \lambda^m \circ \mathbb{D}$ for $m \geq 1$. We need to show that they are equal, i.e. that the two group homomorphisms
\begin{align*}
    \lambda_t, (\lambda')_t: U^{-1}\Kstk \rightarrow (1+U^{-1}\Kstk[[t]])^\times
\end{align*}
coincide. 

By Remark \ref{lambda extn rmk}, both $\lambda_t$ and $\lambda'_t$ are $T^{-1}\mathbb{Z}[\tau, \mathbb{L}]$-linear. As $U^{-1} \Kstk \cong T^{-1} \Kdim$ is generated by $\Kspr$ as a $T^{-1} \mathbb{Z}[\tau, \mathbb{L}]$-module, we need only check that they agree on the class $\{ X \}$ of a smooth and proper variety $X$ over $k$.

This is true, as Corollary \ref{DM in spr} says that both $\{ X \}$ and $\{ \Symm^m X \}$ are fixed by the involution $\mathbb{D}$.
\end{proof}
We upgrade this Proposition to a local result in Theorem $\ref{sym d-sing}$, in the context of $\mathbb{D}$-singularities.

\section{\texorpdfstring{$\mathbb{D}$}{��}-Singularities}\label{D sing section}

In this section, we discuss a class of singularities which are well-behaved with respect to the graded Grothendieck group.

\subsection{Definition and Properties}
Let $X$ be a base scheme.
\begin{defn}
We say $X$ has $\mathbb{D}$-singularities if the identity morphism $\mathbbm{1}_X = [X \xrightarrow[]{id} X]$ of $\KdimX$ lies in the submodule $\KsprX$.
\end{defn}

Equivalently, $X$ has $\mathbb{D}$-singularities if $\pi_2(\mathbbm{1}_X) = 0$ in $\KsprX$. 

More explicitly, $X$ has $\mathbb{D}$-singularities if and only if it can be built by cutting into pieces a $\mathbb{Z}$-linear combination of regular schemes, proper over $X$ of relative dimension zero, and then gluing back together.

Our notion of $\mathbb{D}$-singularities refines ad hoc definitions in the literature, when $X$ is a variety over $k$.

\begin{exmp}\label{L exmp}
Recall that $X$ is said to have $\mathbb{L}$-rational singularities \cite{stablebirational} if there exists a resolution of singularities $f: \tilde{X} \rightarrow X$ such that, for every point $x \in X$, we have $[f^{-1}(x)] \equiv 1_{k(x)}$ mod $\mathbb{L}$ in $K_0(\Var_{k(x)})$. 

Let us show that $\mathbb{D}$-singularities imply $\mathbb{L}$-rational singularities. If $f: \Tilde{X} \rightarrow X$ is a resolution, then
\begin{align*}
    [X \rightarrow X] = [\Tilde{X} \rightarrow X] + \tau \beta
\end{align*}
for some class $\beta \in \KdimX$. We may uniquely write $\beta = \gamma+\mathbb{L} \alpha$ for classes $\gamma, \alpha \in \KsprX$. Under the assumption $X$ has $\mathbb{D}$-singularities, we deduce $\pi_2(\tau \beta) = \gamma= 0$, hence 
\begin{align*}
    [X \rightarrow X] = [\Tilde{X} \rightarrow X] + \tau \mathbb{L} \alpha.
\end{align*}
Pulling back along any point $x: \mathop{\mathrm{Spec}}k(x) \rightarrow X$, we obtain $[f^{-1}(x)] \equiv 1_{k(x)}$ modulo $\mathbb{L}$ in $K_0(\Var_{k(x)})$. Hence $X$ has $\mathbb{L}$-rational singularities.
\end{exmp}

\begin{exmp}\label{B exmp}
The authors of \cite{kontsevichtschinkel} associate an invariant $\partial_Z(X)$ to each closed subvariety $Z$ of $X$ containing the singular locus of $X$. Using our terminology, it is immediate to check that one can identify $\partial_Z(X)$ with
\begin{align*}
    -i^!(\pi_2 [U \rightarrow X]) \in K_0(\Var_Z^{\dim})/(\tau),
\end{align*}
where $U = X-Z$ and $i:Z \rightarrow X$. Assume that $Z$ is an irreducible divisor. Then, in this terminology, the pair $(X,Z)$ is said to have B-rational singularities if $[Z] = -i^!(\pi_2 [U \rightarrow X]) \in K_0(\Var_Z^{\dim})/(\tau)$.

If $Z$ has $\mathbb{D}$-singularities, then one checks $[Z \rightarrow X] = -\pi_2[U \rightarrow X] \in \KdimX$, so $(X, Z)$ indeed has B-rational singularities.
\end{exmp}

\begin{prop}\label{local}
Let $X$ be a $k$-scheme. Then the following are equivalent:
\begin{enumerate}
    \item $X$ has $\mathbb{D}$-singularities,
    \item there exists an open cover $\cup_{i \in I} U_i$ of $X$ such that $U_i$ has $\mathbb{D}$-singularities for each $i \in I$,
    \item $\mathop{\mathrm{Spec}}\mathscr{O}_{x,X}$ has $\mathbb{D}$-singularities for each closed point $x \in X$.
\end{enumerate}
\end{prop}

\begin{proof}
By spreading out, Theorem \ref{Spreading Out}, we see that $\mathop{\mathrm{Spec}}\mathscr{O}_{x,X}$ has $\mathbb{D}$-singularities if and only if $U$ has $\mathbb{D}$-singularities for some (affine) open subset $U$ containing $s$. Thus properties $(2)$ and $(3)$ are equivalent. Clearly $(1)$ implies $(2)$. For the implication $(2)$ implies $(1)$, we use inclusion-exclusion:
\begin{align*}
    \pi_2(\mathbbm{1}_X)= \sum_{\emptyset \neq J \subset I} (-1)^{|J|-1} (j_{U_J})_! j_{U_J}^* \pi_2(\mathbbm{1}_{X}).
\end{align*}
Since $\pi_2$ commutes with pullback along a smooth morphism, we have
\begin{align*}
    \pi_2(\mathbbm{1}_X)= \sum_{\emptyset \neq J \subset I} (-1)^{|J|-1} (j_{U_J})_! \pi_2(\mathbbm{1}_{U_J}).
\end{align*}
Each term on the right side is zero, hence so is $\pi_2(\mathbbm{1}_X)$.
\end{proof}

Let's look at some example to get a better sense of the meaning of $\mathbb{D}$-singularities for complex varieties $X$.

\begin{exmp}
If $X$ is smooth, it has $\mathbb{D}$-singularities.
\end{exmp}

\begin{exmp}\label{curve example}
If $C$ is a nodal curve with a node at $p \in C$, then
\begin{align*}
    [C] = [\tilde{C}] - \tau \cdot [p],
\end{align*}
where $\tilde{C} \rightarrow C$ is the normalization. Thus $\pi_2[C] = - [p] \neq 0$, so $C$ does not have $\mathbb{D}$-singularities. 

On the other hand, if $C$ has a cusp at $p \in C$, then $[C] = [\tilde{C}]$, so $C$ has $\mathbb{D}$-singularities.

More generally, a curve $C$ has $\mathbb{D}$-singularities if and only if it is unibranched at each point.
\end{exmp}

\begin{prop}
Suppose a variety $X$ of dimension $n+1$ has isolated singularities, supported at a point $p \in X$. Choose a log resolution $f: \tilde{X} \rightarrow X$, so $f^{-1}(p) = \cup_{i \in I} d_i D_i$ is a snc divisor. Then $X$ has $\mathbb{D}$-singularities if and only if
\begin{align*}
    \sum_{\emptyset \neq J \subset I} (-1)^{|J|-1} [D_J \times \mathbb{P}^{|J|-1}] = [\mathbb{P}^n] \in K_0(\Var_{\mathbb{C}}^{\dim}).
\end{align*}
\end{prop}

\begin{proof}
We compute, by inclusion-exclusion,
\begin{align*}
    [X] = \sum_{J \subset I} (-1)^{|J|} \tau^{|J|} [D_J] + \tau^{n+1},
\end{align*}
hence
\begin{align*}
    -\pi_2[X] = \sum_{\emptyset \neq J \subset I} (-1)^{|J|-1} [D_J \times \mathbb{P}^{|J|-1}] - [\mathbb{P}^n].
\end{align*}
\end{proof}

\begin{exmp}
Let $X$ be a cone over a smooth proper $k$-variety $Y$ of dimension $n$. Then $X$ has $\mathbb{D}$-singularities if and only if $[Y] = [\mathbb{P}^{n}]$ in $K_0(\Var^{\dim}_\mathbb{C})$.

For instance, a cone over a Veronese embedding of $\mathbb{P}^n$ or over an odd-dimensional quadric has $\mathbb{D}$-singularities.

Moreover, a cone over a smooth stably rational variety whose class is not that of $\mathbb{P}^n$ provides an example of a variety with $\mathbb{L}$-rational but not $\mathbb{D}$-singularities.
\end{exmp}

\begin{exmp}
Du Val surface singularities are $\mathbb{D}$-singular, as can be seen from explicit resolutions.
\end{exmp}

\begin{exmp}
Let $Z \subset \mathbb{P}^{n-1}$ be a smooth hypersurface of degree $d \leq n$. Embed $\mathbb{P}^{n-1}$ in $\mathbb{P}^n$ as a hyperplance and consider the blow-up $Y = \Bl_Z \mathbb{P}^n$ of $\mathbb{P}^n$ along $Z$.

The class $\ell-d e \in \overline{\NE}(Y)$ of the strict transform of a line in $\mathbb{P}^{n-1}$ pairs to $K_Y$ with intersection number
\begin{align*}
    K_Y \cdot \big(\ell-d e \big) = \big(-(n+1)H+E \big) \cdot \big(\ell  - d e \big) = -(n+1) + d < 0.
\end{align*}
This implies there is a contraction morphism $\cont_{\ell- d e}: Y \rightarrow X$. We view $X$ as a compactification of $\mathbb{A}^n$, with boundary the cone over $Z$ with vertex $p$. In fact, $X$ has an isolated singularity at $p$, locally isomorphic to a cyclic quotient $\mathbb{A}^n / \mu_{d-1}$.

We show that $X$ has $\mathbb{D}$-singularities. Indeed, inside $K_0(\Var_X^{\dim})$,
\begin{align*}
    [X] =& [Y] - \tau \cdot [\mathbb{P}^{n-1} \rightarrow p] + \tau^n [p \rightarrow p] \\
    =& [Y] - \tau \mathbb{L} \cdot [\mathbb{P}^{n-2} \rightarrow p].
\end{align*}
\end{exmp}

\begin{exmp}\label{rhm example}
One can show directly that:
\begin{enumerate}
    \item $X = ( x_1^2 + \ldots + x_{n-1}^2 + x_n^d = 0 ) \subset \mathbb{A}^n$ has $\mathbb{D}$-singularities if and only if either $n$ or $d$ is odd;
    \item $X = ( x_1^2 + \ldots + x_{n-2}^2 + x_{n-1}^3 + x_n^3 = 0 ) \subset \mathbb{A}^n$ has $\mathbb{D}$-singularities if and only if $n$ is odd;
    \item $X = ( x_1^{d-1} + \ldots + x_{n-1}^{d-1} + x_n^d = 0 ) \subset \mathbb{A}^n$ always has $\mathbb{D}$-singularities.
\end{enumerate}

In this example, $X$ has $\mathbb{D}$-singularities if and only if $X$ is a rational homology manifold (see Definition \ref{rhm defn}). Indeed, Milnor \cite[Sections $7$ and $8$]{milnor} shows that $(x_1^{a_1} + \ldots x_n^{a_n} = 0)$ is a rational homology manifold if and only if there is no integer of the form 
\begin{align*}
    \frac{b_1}{a_1} + \ldots + \frac{b_n}{a_n}, \hspace{.4cm} \mathrm{where} \hspace{.2cm} b_i \in \{ 1, 2,  \ldots , a_i-1 \}
\end{align*}
\end{exmp}

\begin{defn}\label{rhm defn}
Let $X$ be a complex algebraic variety of dimension $n$. Say $X$ is a rational homology manifold if for any point $x \in X$,
\begin{align*}
    H^i_{\{x\}}(X, \mathbb{Q}) =  
    \begin{cases}
                \mathbb{Q} & \text{if } i = 2n \\
                0 & \text{otherwise}
            \end{cases}
\end{align*}
\end{defn}
A rational homology manifold should be thought of as being smooth for the purposes of rational (co)homology theories. For example, rational homology manifolds satisfy Poincaré duality.

There is a relationship between $\mathbb{D}$-singularities and rational homology manifolds more generally.

\begin{prop}\label{D-sing implies rhm}
Suppose $X$ is a local complete intersection with isolated singularities. If $X$ has $\mathbb{D}$-singularities, then $X$ is a rational homology manifold.
\end{prop}

\begin{proof}
We may assume $X$ has dimension $n > 1$ by Example \ref{curve example}. Let $i_x: \mathop{\mathrm{Spec}}\mathbb{C} \rightarrow X$ be a singular point. Using the notation of Subsection \ref{mhm subsection}, we compute
\begin{align*}
    [i_x^! \mathbb{Q}_X^H] &= i_x^! \chi \mathbbm{1}_X  \\
                        &= \mathbb{D} i_x^* \mathbb{D} \chi \mathbbm{1}_X \\
                        &= (\mathbb{D} \chi i_x^* \mathbb{D} \mathbbm{1}_X)(-n) \\
                        &= [\mathbb{Q}^H(-n)]
\end{align*}
inside $K_0(\MHS(\mathbb{C}))$, where the last step uses our assumption that $X$ has $\mathbb{D}$-singularities.

By \cite{hamm}, $H^i(i_x^! \mathbb{Q}_X) = H^i_{\{x\}}(X, \mathbb{Q}) = 0$ unless $i = n, n+1$, or $i = 2n$, as $X$ is a local complete intersection. Combining this with the above, we deduce
\begin{align*}
    [H^{n}_{\{x\}}(X, \mathbb{Q})] = [H^{n+1}_{\{x\}}(X, \mathbb{Q})] \in K_0(\MHS(\mathbb{C})).
\end{align*}
But $H^{n}_{\{x\}}(X, \mathbb{Q})$ has weights $\leq (n-1)$ and $H^{n+1}_{\{x\}}(X, \mathbb{Q})$ has weights $\geq (n+1)$ by \cite{link}. We conclude that $H^{n}_{\{x\}}(X, \mathbb{Q}) = H^{n+1}_{\{x\}}(X, \mathbb{Q}) = 0$, and hence $X$ is a rational homology manifold.
\end{proof}

Nevertheless, there is no direct implication between $\mathbb{D}$-singularities and rational homology manifolds. Indeed, related to the failure of Noether's question, there exist quotient singularities (hence rational homology manifolds) which are not $\mathbb{D}$-singularities; see \cite{esserquotient} for details. See also the following example.

\begin{exmp}\label{rhm example}
Let $S$ be a surface with an isolated cuspidal singularity, i.e. such that the exceptional divisor of the minimal resolution is a cycle of rational curves. Let $T$ be the surface obtained by gluing the singular point of $S$ to a smooth point (see \cite[Section $3$]{schwede}). It is straightforward to compute that $T$ has $\mathbb{D}$-singularities, but it is not a rational homology manifold, as the link of the singular point is disconnected.
\end{exmp}

\subsection{Toric Orbifolds}
The following Proposition is crucial in the proof of the more general Thereom \ref{abelian sing}.

\begin{prop}\label{toric}
Let $X$ be a simplicial toric variety (or, equivalently, a toric variety with finite abelian quotient singularities). Then $X$ has $\mathbb{D}$-singularities.
\end{prop}

This result is essentially proved by Bittner \cite[Corollary $4.2$]{bittnertoric}, using different language. For the convenience of the reader, we will sketch the proof. It relies on the following combinatorial lemma, which is a consequence of the Dehn-Somerville equations.

\begin{lem}\label{b}[cf. \cite[Lemma $4.6$]{bittnertoric}]\label{sym toric lem}
Let $\Sigma$ be a simplicial refinement of the cone $\Delta$ in $\mathbb{R}^n$ spanned by the standard basis vectors. For each cone $\sigma \in \Sigma$, let $\phi(\sigma) \in \Delta$ be the smallest cone in $\Delta$ containing $\sigma$. Then, denoting by $|\sigma|$ the dimension of the simplex $\sigma$, the degree $n$ polynomial
\begin{align*}
    p_\Delta(s,t) = \sum_{\tau \in \Delta} \sum_{\phi(\sigma) = \tau} (-1)^{|\tau|} (s-t)^{|\tau|-|\sigma|} t^{n-|\tau|+|\sigma|},
\end{align*}
is symmetric in $s$ and $t$.
\end{lem}

If $Y$ is a toric variety defined by a fan $\Sigma$ of dimension $n$ and if $\sigma \in \Sigma$, denote by $\mathscr{O}_\sigma$ the open orbit corresponding to $\sigma$ and by $V_\sigma$ its closure. From the identity $V_\sigma = \coprod_{\sigma \subset \sigma'} \mathscr{O}_{\sigma'}$, we deduce, by inclusion-exclusion,
\begin{align*}
    [\mathcal{O}_\sigma] = \sum_{\sigma \subset \sigma'} (-\tau)^{|\sigma'|-|\sigma|} [V_{\sigma'}].
\end{align*}

\begin{proof}[Proof of Proposition \ref{toric}]
Take a toric resolution $Y \rightarrow X$ of $X$. Denote by $\Sigma$ and $\Delta$ the fans defining $Y$ and $X$ respectively. For each cone $\sigma \in \Sigma$, denote by $\phi(\sigma) \in \Delta$ the smallest cone in $\Delta$ containing $\sigma$. Each $\sigma \in \Sigma$ determines a $\mathbb{G}_m^{|\phi(\sigma)|-|\sigma|}$-bundle $\mathcal{O}_\sigma \rightarrow \mathcal{O}_{\phi(\sigma)}$. We thus compute
\begin{align*}
    [Y] =& \sum_{\sigma \in \Sigma} [\mathcal{O}_\sigma] \tau^{|\sigma|} \\
        =& \sum_{\tau \in \Delta} \sum_{\phi(\sigma) = \tau} (\mathbb{L}-\tau)^{|\tau|-|\sigma|} [\mathcal{O}_{\phi(\sigma)}] \tau^{|\sigma|} \\
        =& \sum_{\tau' \in \Delta} \bigg( \sum_{\tau \subset \tau'} \sum_{\phi(\sigma) = \tau } (-1)^{|\tau|} (\mathbb{L}-\tau)^{|\tau|-|\sigma|} \tau^{|\tau'|-|\tau|+|\sigma|} \bigg) (-1)^{|\tau'|} [V_{\tau '}] \\
        =& \sum_{\tau' \in \Delta, |\tau'| \geq 1} \bigg( \sum_{\tau \subset \tau'} \sum_{\phi(\sigma) = \tau} (-1)^{|\tau|} (\mathbb{L}-\tau)^{|\tau|-|\sigma|}  \tau^{|\tau'|-|\tau|+|\sigma|} \bigg) (-1)^{|\tau'|} [V_{\tau'}] + [X] 
\end{align*}
This last line is equal to 
\begin{align*}
    \sum_{\tau' \in \Delta, |\tau'| \geq 1} (-1)^{|\tau'|} p_{\tau'}(\mathbb{L}, \tau) [V_{\tau'}] + [X].
\end{align*}
by Lemma \ref{sym toric lem}.

But $V_{\tau'}$ is a simplicial toric variety of dimension less than $n$, so we can assume by induction on $n$ that it has $\mathbb{D}$-singularities. Combined with Lemma \ref{sym toric lem}, which says that $p_{\tau'}(\mathbb{L}, \tau)$ is symmetric in $\tau$ and $\mathbb{L}$, we see that
\begin{align*}
    [X] = [Y] - \sum_{\tau' \in \Delta, |\tau'| \geq 1} (-1)^{|\tau'|} p_{\tau'}(\mathbb{L}, \tau) [V_{\tau'}]
\end{align*}
is a sum of terms, all of which lie in $\KsprX$. Hence $X$ has $\mathbb{D}$-singularities.
\end{proof}

\subsection{Abelian Quotient Singularities}

\begin{thm}\label{abelian sing}[cf. Theorem \ref{intro abelian quot}]
Let $G$ be a finite abelian group acting on a smooth quasi-projective variety $X$ over a field $k$ of characteristic zero. Then the quotient variety $X/G$ has $\mathbb{D}$-singularities. 
\end{thm}

For the proof, we will require some setup.

\begin{lem}\label{vector space}
Let $V$ be an $n$-dimensional vector space over a field $k$ of characteristic zero and let $G$ be a finite abelian group acting linearly on $V$. Then
\begin{align*}
    [\mathbb{P}(V)/G] = [\mathbb{P}^{n-1}] \in K_0(\Var_k^{\dim}).
\end{align*}
\end{lem}

\begin{proof}
Diagonalize the action of $G$ on $V$. This provides a decomposition of $V- \{ 0 \}$ as a disjoint union of $G$-equivariant standard coordinate tori $\mathbb{G}_{m}^{\ell+1}$. Taking the quotient by the scalar action of $\mathbb{G}_{m}$, we obtain a $G$-equivariant decomposition of $\mathbb{P}V$ into tori $\mathbb{G}_{m}^\ell$, with $G$ acting diagonally. Each of the quotients of the tori $\mathbb{G}_{m}^\ell$ by the $G$-action remain isomorphic to $\mathbb{G}_{m}^\ell$. Therefore $[\mathbb{P}(V)/G] = [\mathbb{P}V]$.
\end{proof}

\begin{lem}\label{pointwise}
Let $Z$ be a variety over a field $k$ of characteristic zero and let $\alpha \in K_0(\Var_Z^{\dim})$. Then $\alpha = 0$ if $i_p^* \alpha = 0$ in $K_0(\Var_{k(p)}^{\dim})$ for every (possibly non-closed) point $i_p: \mathop{\mathrm{Spec}}k(p) \rightarrow Z$.
\end{lem}

\begin{proof}
We induct on the dimension of $Z$. First reduce to the case where $Z$ is irreducible by considering the components separately. Now let $\eta \in Z$ be the generic point. Then $i_\eta^* \alpha = 0$ implies by spreading out, Theorem \ref{Spreading Out}, that $j_U^* \alpha = 0$ for some open subset $U$ of $Z$. If $W \subset Z$ is the complement of $U$, then $\alpha = (i_W)_! \beta$ for some $\beta \in K_0(\Var_W^{\dim})$, which vanishes when pulled back to any point. By the inductive hypothesis, $\beta = 0$, and hence we conclude that $\alpha = 0$.
\end{proof}

\begin{prop}\label{514}
Let $G$ be a finite abelian group acting on a quasi-projective variety $X$ over a field $k$ of characteristic zero, and let $E$ be a $G$-equivariant vector bundle of rank $n$ over $X$. Then
\begin{align*}
    [\mathbb{P}(E)/G ] = [\mathbb{P}^{n-1}] \cdot [X/G] \in K_0(\Var_{X/G}^{\dim}).
\end{align*}
\end{prop}

\begin{proof}
Denote by $\pi: X \rightarrow X/G$ the quotient map. Let $p \in X/G$ be a point and denote $W = \pi^{-1}(p)$. By Lemma \ref{pointwise}, it suffices to show that $[\mathbb{P}E_W/G] = [\mathbb{P}_{k(p)}^{n-1}]$ in $K_0(\Var_{k(p)}^{\dim})$, where $E_W$ is the restriction of $E$ to $W$.

If we pick a point $x \in W$, we have $\mathbb{P}E_W/G \cong \mathbb{P}E_x / G_x$, where $G_x \subset G$ is the stabilizer of $x \in W$ and $E_x$ is the restriction of $E$ to $x$. By Lemma \ref{vector space} applied to $E_x$, considered as a vector space over the field $k(p) = k(x)^{G_x}$, we indeed find
\begin{align*}
    [\mathbb{P}E_W/G] = [\mathbb{P}E_x/G_x] = [\mathbb{P}_{k(p)}^{n-1}] \in K_0(\Var_{k(p)}^{\dim}).
\end{align*}
\end{proof}

\begin{cor}\label{G-equ weak}
Let $G$ be a finite abelian group acting on a smooth quasi-projective variety $X$ over a field $k$ of characteristic zero, and let $f:\Bl_Y X \rightarrow X$ be a $G$-equivariant blow-up along a smooth closed subvariety $Y \subset X$ of codimension $k+1$. Then
\begin{align*}
    [\Bl_Y X /G]= [X/G] + \tau \mathbb{L} \cdot  [Y/G ] \cdot [\mathbb{P}^{k-1}] \in K_0(\Var_{X/G}^{}).
\end{align*}
\end{cor}

\begin{proof}
We compute
\begin{align*}
    [\Bl_Y X/G] - [X/G] =& \tau \cdot [\mathbb{P}(N_{Y / X})] - \tau^{k+1} \cdot [Y/G] \\
    =& [Y/G] ( \tau \cdot [\mathbb{P}^{k}] - \tau^{k+1} ) \\
    =& \tau \mathbb{L} \cdot [Y/G] \cdot [\mathbb{P}^{k-1}],
\end{align*}
using Proposition \ref{514}.
\end{proof}

Let us also cite the following result.

\begin{thm}\label{toroidally}[cf. \cite[Theorem $0.1$]{abramovichwang}, \cite[Theorem $3.3$]{esserquotient}]
Let $X$ be a projective variety over $k$, equipped with an action by a finite group $G$. Then there is a $G$-equivariant proper birational morphism $r:X_1 \rightarrow X$ and an open set $U \subset X_1$ such that $X_1$ is a nonsingular projective variety, $(X_1, U)$ has the structure of a strictly toroidal embedding, and $G$ acts strictly and toroidally on $(X_1, U)$.
\end{thm}

In particular, by \cite[Proposition $2.13$]{esserquotient}, the quotient $X_1/G$ satisfies the following property: for all points $p \in X_1/G$, there exists an open neighborhood $V$ of $p$ and an étale morphism $V \rightarrow T$, where $T$ is a simplicial toric variety. 

As $T$ has $\mathbb{D}$-singularities by Proposition \ref{toric}, so too does $V$ by Proposition \ref{smooth}. But as the property of having $\mathbb{D}$-singularities can be checked locally by Proposition \ref{local}, we conclude that $X_1/G$ has $\mathbb{D}$-singularities.

\begin{proof}[Proof of Theorem $5.9$]
We argue by induction on $\dim X$. Since any quasi-projective $G$-variety $X$ admits a $G$-equivariant compactification and the property of having $\mathbb{D}$-singularities is local, we may as well assume $X$ is projective.

Take a $G$-equivariant proper birational morphism $r:X_1 \rightarrow X$ as in Theorem \ref{toroidally}. By the $G$-equivariant weak factorization theorem \cite{weakfactorization}, $r$ is a composition of blow-ups and blow-downs of smooth proper subvarieties, all of which are defined over $X$. Applying Corollary \ref{G-equ weak} repeatedly, we see that $[X/G] \in K_0(\Var_{X/G}^{\dim})$ is a sum of $[X_1/G]$ with classes of the form $\tau \mathbb{L} \cdot [Y/G] \cdot [\mathbb{P}^{k-1}]$, where $Y$ is a smooth $G$-variety with $\dim Y < \dim X$. Our induction hypothesis says that $Y/G$ has $\mathbb{D}$-singularities and hence $\tau \mathbb{L} \cdot [Y/G] \cdot [\mathbb{P}^{k-1}] \in K_0(\Var_{X/G}^{\spr})$. We are thus reduced to checking that $X_1/G$ has $\mathbb{D}$-singularities. This follows from the above discussion.
\end{proof}

\subsection{Symmetric Powers}\label{sym section}

\begin{thm}\label{sym d-sing}
Let $X$ be a smooth variety over $k$. Then the class
\begin{align*}
    \{ \Sym^m X \rightarrow \Sym^m X \} \in K_0(\Stk_{\Sym^m X}^{\dim})
\end{align*}
lies in the submodule $T^{-1}K_0(\Var_{\Sym^m X}^{\spr})$, following the notation of \ref{DM section}. Hence 
\begin{align*}
    \pi_2(\mathbbm{1}_{\Sym^m X}) \in \ker( K_0(\Var_{\Sym^m X}^{\spr}) \rightarrow T^{-1}K_0(\Var_{\Sym^m X}^{\spr})),
\end{align*}
i.e. $\Sym^m X$ has $\mathbb{D}$-singularities, up to a $T$-torsion element of $K_0(\Var_{\Sym^m X}^{\spr})$.
\end{thm}

\begin{proof}
By Remark \ref{sym = symm},
\begin{align*}
    \{ \Sym^m X \} = \{ \Symm^m X \} \in K_0(\Var_{\Sym^m X}^{\dim}).
\end{align*}
As $\Symm^m X = [X^m/S_m]$ is a smooth Deligne-Mumford stack, we conclude $\{ \Symm^m X \} \in K_0(\Var_{\Sym^m X}^{\spr})$ by Corollary \ref{DM in spr}.
\end{proof}

The following conjecture remains open.

\begin{conj}\label{is sym D-sing?}
Let $X$ be a smooth variety and let $m \geq 1$. Then the symmetric power $\Sym^m X$ has $\mathbb{D}$-singularities.
\end{conj}

Symmetric powers of curves are smooth. If $X$ is a surface, we obtain a positive answer by the formula of \cite{gottschehilbert} for the resolution $\Hilb^m X \rightarrow \Sym^m X$.

In arbitrary dimension, \cite{shein} shows the strictly weaker result that symmetric powers have $\mathbb{L}$-rational singularities.

\section{Kapranov Zeta Functions}\label{kapranov zeta section}

\subsection{Irrationality Result}

\begin{thm}\label{irrational}[cf. Theorem \ref{intro Kapranov}]
Let $X$ be a smooth projective variety of dimension $>1$ and non-negative Kodaira dimension $\kappa(X) \geq 0$. Then the graded Kapranov zeta function $Z_X^{\dim}(t) \in K_0(\Var_k^{\dim})[[t]]$ of $X$ is (pointwise) irrational. 

In fact, the ungraded Kapranov zeta function $Z_X(t) \in K_0(\Var_k)[[t]]$ of $X$ is also (pointwise) irrational, if we assume a positive answer to Question \ref{is sym D-sing?}. 
\end{thm}

If $R$ is a ring, we say a power series $Z(t) \in R[[t]]$ is pointwise irrational if there is a ring homomorphism $\mu: R \rightarrow F$, where $F$ is a field, such that $\mu(Z(t)) \in F[[t]]$ is irrational.

We focus on the graded Kapranov zeta function $Z_X^{\dim}(t) \in K_0(\Var_k^{\dim})[[t]]$. We first give a simpler proof of Theorem \ref{irrational} when $\kappa(X) \geq 1$. Then we prove the general case by making use of Adams operations, and we discuss the adjustments for the ungraded version, assuming a positive answer to Question \ref{is sym D-sing?}.

Larsen and Lunts proved that the Kapranov zeta function $Z_X(t) \in K_0(\Var_k)[[t]]$ of a smooth projective surface of Kodaira dimension $\kappa(X) \geq 0$ is pointwise irrational in \cite{larsenluntsrationality}. We follow their line of argument, which extends to arbitrary dimension by crucially using the following result of Arapura-Archava \cite{arapurakodaira}.

\begin{thm}\label{plurigenera}
Let $X$ be a smooth projective variety of dimension $n>1$. If $Z_m$ is a resolution of $\Sym^m X$, then
\begin{align*}
    H^0(Z_m, \omega_{Z_m}^{\otimes n}) \cong \Sym^m H^0(X, \omega_X^{\otimes d}),
\end{align*}
whenever $nd$ is even.
\end{thm}

\subsection{Proof when $\kappa(X) \geq 1$}
Choose $d \geq 1$ so that $h= h^0(X, \omega_X^{\otimes d}) \geq 2$. Replacing $d$ by $2d$ if necessary, we assume that $d$ is even.

Consider the map of monoids 
\begin{align*}
   \Bir_k \rightarrow \mathbb{N}, \hspace{.5cm} [Z] \mapsto h^0(Z, \omega_Z^{\otimes d}).
\end{align*}
By Theorem \ref{plurigenera}, this map commutes with $\lambda$-operations. Hence the corresponding map of group rings $\mathbb{Z}[\Bir_k] \rightarrow \mathbb{Z}[\mathbb{N}]$ is a $\lambda$-homomorphism. Consider now the composite $\lambda$-homomorphism
\begin{align*}
\mu = \mu_d: K_0(\Var_k^{\dim}) \rightarrow \mathbb{Z}[\Bir_k] \rightarrow \mathbb{Z}[\mathbb{N}].
\end{align*}
We show that
\begin{align*}
    \mu(Z_X^{\dim}(t)) = \sum_{m \geq 0} \bigg[ \binom{m+h-1}{h-1} \bigg] t^m.
\end{align*}
is not rational in $F[[t]]$, where $F$ is the fraction field of $\mathbb{Z}[\mathbb{N}]$.

We use the following Lemma.

\begin{lem}\label{a}[cf. \cite[Section $3$]{larsenluntsstable}]
If $F$ is a field, a power series $\sum_m a_m t^m$ is rational if and only if there exists $j>0$ and $j_0>0$ such that for each $m>j_0$, the determinant of the matrix
\[
\begin{bmatrix} 
a_{m} & \dots  & a_{m+j}\\
\vdots & \ddots & \vdots\\
a_{m+j} & \dots  & a_{m+2j} 
\end{bmatrix}
\]
is zero.
\end{lem}

Spelling this out, we need to show that for each $j>0$, there exists $m$ sufficiently large so that

\begin{align*}
\sum_{\sigma \in S_{j+1}} \sgn(\sigma) \bigg[  \prod_{i=1}^{j+1} \binom{m+\sigma(i)+i+h-3}{h-1} \bigg] \neq 0.
\end{align*}

Each of the summands on the left hand side, considered as a polynomial in $m$, determines the multiset $\{ \sigma(i) + i \}$. The multiset corresponding to the identity permutation is clearly unique, so there cannot be any cancellation in the monoid ring $\mathbb{Z}[\mathbb{N}]$ for all $m$. The result follows.

Note that this proof uses the group ring $\mathbb{Z}[\mathbb{N}]$ crucially. The power series
\begin{align*}
    \sum_{m \geq 0} \binom{m+h-1}{h-1} t^m = \frac{1}{(1-t)^h}
\end{align*}
is certainly rational in $\mathbb{Z}[[t]]$.

Moreover, this proof fails completely in the ungraded situation, as $K_0(\Var_k)$ only admits $\mathbb{Z}[\SB_k]$ as a quotient, and plurigenera are not stably birational invariants.

\subsection{Proof when $\kappa(X) \geq 0$}
To attack the general case of Theorem \ref{irrational}, we will need to keep track of more information. We use the Adams operations on the Grothendieck group $\overline{K}(X)$ as in Example \ref{K bar}, together with the global sections map
\begin{align*}
    h^0: \overline{K}(X) \rightarrow \mathbb{Z}.
\end{align*}

Fix $d \geq 1$ and denote $h_d^i(Z) = h^0(\Psi^d [\Omega_Z^{i}])$ if $Z$ is a smooth projective variety.

Let $M \subset \mathbb{Z}[s]$ be the monoid of polynomials with constant term one under multiplication, and consider the map $\nu = \nu_d$ which assigns to a smooth projective variety $Z$, the element
\begin{align*}
    \nu(Z) = 1+h_d^1(Z)s + \ldots +h_d^n(Z)s^d \in M,
\end{align*}
where $n = \dim Z$. Note that $\Psi^d \omega_Z = \omega_Z^{\otimes d}$, so the highest degree coefficient of $\nu(Z)$ is the plurigenus $h^0(Z, \omega_Z^{\otimes d})$.

The following is proved in \cite[Proposition $6.1$]{larsenluntsrationality}.

\begin{lem}\label{mu defined}
The map $\nu$ satisfies the following properties:
\begin{enumerate}
    \item $\nu(Z) = \nu(Z')$ if $Z$ and $Z'$ are birational,
    \item $\nu(Z \times Z') = \nu(Z) \nu(Z')$, 
    \item $\nu(\mathbb{P}^1)=1$.
\end{enumerate}
\end{lem}
In particular, we can view $\nu$ as a map of monoids $\Bir_k \rightarrow M$.

Consider now the composite ring homomorphism
\begin{align*}
    \mu = \mu_d: \Kdim \rightarrow \mathbb{Z}[\Bir_k] \rightarrow \mathbb{Z}[M].
\end{align*}

\begin{rmk}
The monoid ring $\mathbb{Z}[M]$ admits a natural $\lambda$-structure, but $\mu$ is not a $\lambda$-homomorphism for $n > 1$.
\end{rmk}

Now let's prove Theorem \ref{irrational}. As before, choose an even integer $d$ such that $h=h^0(X, \omega_X^{\otimes d}) \geq 1$.

Note that $M$ is a free abelian monoid, so $\mathbb{Z}[M]$ is an integral domain with fraction field $F$. We show that the image $f(t)$ of
\begin{align*}
    \sum_{m \geq 0} \mu([\Sym^m X]) t^m = \sum_{m \geq 0} \mu([Z_m]) t^m
\end{align*}
in $F[[t]]$ is irrational, where $Z_m$ is a resolution of $\Sym^m X$.

We now require a few lemmas from \cite{larsenluntsrationality}. The first is proved only for the resolution $\Hilb^m X \rightarrow \Sym^m X$ when $X$ is a surface, but the argument goes through in general.

\begin{lem}\label{bounded}[cf. \cite[Proposition $7.5$]{larsenluntsrationality}]
Fix $d \geq 1$ and $i \geq 1$. Then $h_d^i(Z_m) = h^0(\Psi^d \Omega_{Z_m}^i)$ is bounded, independent of $m$.
\end{lem}

\begin{lem}\label{rationality criteria}[cf. \cite[Proposition $2.6$]{larsenluntsrationality}]
Let $F$ be the field of fractions of the group ring $\mathbb{Z}[G]$ of a free abelian group $G$. Let $f(t) = \sum_{i \geq 0} g_i t^i \in F[[t]]$ be a power series such that, for each $i$, either $g_i \in G$ or $g_i=0$. Then $f(t)$ is rational if and only if there exists $d \geq 1$ and $i_0 \geq 1$, and $\{ h_i \in G \}_{i \in \mathbb{N}}$ periodic with period $d$, such that 
\begin{align*}
    g_{i+d} = h_i g_i
\end{align*}
for all $i \geq i_0$.
\end{lem}

Suppose that $f(t)$ is rational in $F[[t]]$. Apply Lemma \ref{rationality criteria} with $G$ the group completion of $M$. In particular, there exists $h \in G$ and integers $d \geq 1$ and $i_0 \geq 1$ such that
\begin{align*}
    \mu ([Z_{i_0+\ell d}]) = h^\ell \cdot \mu ([Z_{i_0} ])
\end{align*}
for all $\ell \geq 0$.

Note that each $\mu ([Z_m])$ is a polynomial with positive leading term $h^0(Z_m, \omega_{Z_m}^{\otimes d})$. Hence $h$ is a nonconstant polynomial in $s$. But then, looking at the powers of $h$, we see the coefficients of fixed powers of $s$ in the $\mu ([Z_m] )$ cannot all stay bounded, contradicting Lemma \ref{bounded}. This completes the proof of Theorem \ref{irrational} in the graded setting.

The only change we make when working with the ungraded Grothendieck ring $K_0(\Var_k)$ is to use the reduction $r: K_0(\Var_k) \rightarrow \mathbb{Z}[\SB_k]$ modulo $\mathbb{L}$, as opposed to modulo $\tau$. We then define $\mu: K_0(\Var_k) \rightarrow \mathbb{Z}[M]$ by composing $r$ with the morphism $\mathbb{Z}[\SB_k] \rightarrow \mathbb{Z}[M]$, which is well-defined by Lemma \ref{mu defined}. 

However, it is now unclear whether the class $r([\Sym^m X]) \in \mathbb{Z}[\SB_k]$ reflects the stable birational type of $\Sym^m X$, i.e. whether $r([\Sym^m X]) = [Z_m] \in \mathbb{Z}[\SB_k]$. This is not true for a resolution of a general singular variety. We can guarantee it only if we assume that $\Sym^m X$ has $\mathbb{D}$-singularities or, more weakly, that $\Sym^m X$ has $\mathbb{L}$-rational singularities. Shein \cite{shein} recently showed that $\Sym^m X$ does have $\mathbb{L}$-rational singularities and so proved Theorem \ref{irrational} unconditionally in the ungraded setting.

\end{document}